\documentclass[12pt]{article}
\usepackage{graphicx}

\usepackage{amsfonts}
\usepackage{amssymb}
\usepackage{amsmath}
\usepackage{amsthm}
\usepackage{amscd}
\usepackage{pxfonts}
\usepackage{stmaryrd}
\usepackage{amsfonts}
\usepackage{bbm}
\usepackage{amscd}
\usepackage{mathrsfs}
\usepackage{latexsym,amssymb,amsmath,amscd,amscd,amsthm,amsxtra}
\usepackage[T1]{fontenc}
\usepackage{lmodern}
\usepackage{amssymb}
\usepackage[all]{xy}
\usepackage{nicefrac,mathtools,enumitem}
\usepackage{microtype}

\usepackage{enumitem}  
\usepackage{calc}
\setlist{labelindent=1pt,itemsep=.5em}
\setlist[itemize]{leftmargin=1.2cm}
\setlist[enumerate]{itemindent=0em,leftmargin=1.2cm}
\setlist[enumerate,1]{label={\upshape(\roman*)}}

\usepackage{cite}

\usepackage{hyperref}
\hypersetup{
colorlinks   = true,
citecolor    = black,
linkcolor =black,
urlcolor=black,
}

\allowdisplaybreaks[3]

\usepackage{authblk}

\makeatletter
\newcommand{\subjclass}[2][2020]{%
  \let\@oldtitle\@title%
  \gdef\@title{\@oldtitle\footnotetext{#1 \emph{Mathematics subject classification}: #2}}%
}
\newcommand{\keywords}[1]{%
  \let\@@oldtitle\@title%
  \gdef\@title{\@@oldtitle\footnotetext{\emph{Keywords}: #1}}%
}
\makeatother

\makeatletter

\def\sw{\swarrow}
\def\nw{\nwarrow}
\def\ne{\nearrow}
\def\se{\searrow}

\newtheorem{thm}{Theorem}[section]
\newtheorem{prop}{Proposition}[section]
\newtheorem{lem}{Lemma}[section]
\newtheorem{rem}{Remark}[section]

\newtheorem{cor}{Corollary}[section]
\newtheorem{defn}{Definition}[section]

\newtheorem{ex}{Example}[section]

\numberwithin{equation}{section}
\newcommand{\las}{(\!(}
\newcommand{\ras}{)\!)}

\date{}

\title{Hom-pre-Malcev and Hom-M-Dendriform algebras }
\author{Fattoum Harrathi$^{1}$, Sami Mabrouk$^{2}$, Othmen Ncib$^{2}$, Sergei Silvestrov$^{3}$\footnote{{\it Corresponding author}: Sergei Silvestrov}\\
\normalsize{
$^{1}$University of Sfax, Faculty of Sciences Sfax,  BP 1171, 3038 Sfax, Tunisia. \authorcr
harrathifattoum285@gmail.com   \authorcr \vspace{0,2cm}
$^{2}$University of Gafsa, Faculty of Sciences Gafsa, 2112 Gafsa, Tunisia.\authorcr
mabrouksami00@yahoo.fr, othmenncib@yahoo.fr \authorcr \vspace{0,2cm}
$^3$M\"{a}lardalen University,
Division of Mathematics and Physics, \authorcr
School of Education, Culture and Communication, \authorcr
Box 883, 72123 V\"{a}ster{\aa}s, Sweden. \authorcr
sergei.silvestrov@mdh.se}}

\begin{document}
\maketitle

\abstract{The main feature of Hom-algebras is that the identities defining
the structures are twisted by linear maps. The purpose of this paper is to introduce and study a Hom-type generalization of
pre-Malcev algebras and M-dendriform  algebras, called Hom-pre-Malcev algebras and Hom-M-dendriform algebras. We also introduce the notion of $\mathcal{O}$-operators
of Hom-Malcev and Hom-pre-Malcev algebras and
show the connections between Hom-Malcev, Hom-pre-Malcev and Hom-M-dendriform algebras using $\mathcal{O}$-operators. Hom-pre-Malcev algebras and Hom-M-dendriform algebras generalize Hom-pre-Lie algebras and Hom-L-dendriform algebras
respectively to the alternative setting and fit into a bigger
framework with a close relationship with Hom-pre-alternative algebras and Hom-alternative quadri-algebras respectively. \\*[0,3cm]
\noindent\textbf{Keywords}: Hom-pre-Malcev algebra, Hom-M-dendriform algebra, Rota-Baxter operator \\
\noindent\textbf{2020 MSC}: 17D30, 17B61, 17D10, 17A01, 17A30, 17B10}


\newpage

\tableofcontents

\numberwithin{equation}{section}
\section{Introduction}
The notion of dendriform algebras was introduced in 1995 by J.-L. Loday \cite{Loday01} in his study of algebraic $K$-theory. Dendriform algebras
are algebras with two operations, which dichotomize the notion of associative algebras. Later the notion of tridendriform algebra were
introduced by Loday and Ronco in their study of polytopes and Koszul duality (see \cite{LodayRonco04}).
In order to determine the algebraic structures behind a pair of commuting Rota-Baxter operators (on an associative algebra), Aguiar and Loday introduced the notion of quadri-algebra in \cite{AL}.

Malcev algebras were introduced in 1955 by A. I. Malcev \cite{Malcev55:anltcloops}, who called these objects Moufang-Lie algebras because of the connection with analytic Moufang loops. A Malcev algebra is a non-associative algebra $A$ with an anti-symmetric multiplication $[-,-]$ that satisfies, for all $x,y,z \in A$, the Malcev identity
\begin{equation}
\label{malcev}
J(x,y,[x,z]) = [J(x,y,z),x],
\end{equation}
where $J(x,y,z) = [[x,y],z] + [[z,x],y] + [[y,z],x]$ is the Jacobian.  In particular, Lie algebras are examples of Malcev algebras. Malcev algebras play an important role in Physics and the geometry of smooth loops. Just as the tangent
algebra of a Lie group is a Lie algebra, the tangent algebra of a locally analytic Moufang
loop is a Malcev algebra \cite{kerdman,kuzmin68:Malcevalgrepr,kuzmin71:conectMalcevanMoufangloops,Malcev55:anltcloops,nagy,sabinin}, see also \cite{gt,myung,okubo} for discussions about connections with physics.
The notion of pre-Malcev algebra as a Malcev algebraic analogue of a pre-Lie algebra was introduced in \cite{Madariaga}.
A pre-Malcev algebra is a vector space $A$ with a bilinear
multiplication $''\cdot'' $ such that the product $[x,y] = x\cdot y-y \cdot x$ endows
$A$ with the structure of a Malcev algebra, and the left multiplication operator
$L_{\cdot}(x) : y \mapsto x \cdot y$ define a representation of this Malcev algebra on $A$.
In other words, the product  $x \cdot y$ satisfies  the following identities for all $x, y, z, t \in A$:
\begin{equation}\label{PM}
[y,z] \cdot (x \cdot t)+ [[x,y],z] \cdot t+ y \cdot ([x,z] \cdot t)- x \cdot (y \cdot (z \cdot t)) + z \cdot (x \cdot (y \cdot t)).
\end{equation}

In \cite{Madariaga}, in order to find a dendriform algebra whose anti-commutator is a pre-Malcev
algebra, M-dendriform algebras were introduced, an $\mathcal{O}$-operator (specially a Rota-Baxter operator of weight zero) on a pre-Malcev algebra or two commuting Rota-Baxter operators on a Malcev algebra were shown to give a M-dendriform algebra, and also the relationships between M-dendriform algebras and Loday algebras especially quadri-algebras was established.

The theory of Hom-algebras has been initiated in \cite{HartwigLarSil:defLiesigmaderiv, LarssonSilvJA2005:QuasiHomLieCentExt2cocyid,LarssonSilv:quasiLiealg} from Hom-Lie algebras, quasi-Hom-Lie algebras and quasi-Lie algebras, motivated by quasi-deformations of Lie algebras of vector fields, in particular q-deformations of Witt and Virasoro algebras. Hom-Lie algebras and more general quasi-Hom-Lie algebras were introduced first by Hartwig, Larsson and Silvestrov in  \cite{HartwigLarSil:defLiesigmaderiv} where a general approach to discretization of Lie algebras of vector fields using general twisted derivations ($\sigma$-deriva\-tions) and a general method for construction of deformations of Witt and Virasoro type algebras based on twisted derivations have been developed. The general quasi-Lie algebras, containing the quasi-Hom-Lie algebras and Hom-Lie algebras as subclasses, as well their graded color generalization, the color quasi-Lie algebras including color quasi-hom-Lie algebras, color hom-Lie algebras and their special subclasses the quasi-Hom-Lie superalgebras and hom-Lie superalgebras, have been first introduced in \cite{HartwigLarSil:defLiesigmaderiv,LarssonSilvJA2005:QuasiHomLieCentExt2cocyid,LarssonSilv:quasiLiealg,LSGradedquasiLiealg,LarssonSilv:quasidefsl2,SigSilv:CzechJP2006:GradedquasiLiealgWitt}.
Subsequently, various classes of Hom-Lie admissible algebras have been considered in \cite{ms:homstructure}. In particular, in \cite{ms:homstructure}, the Hom-associative algebras have been introduced and shown to be Hom-Lie admissible, that is leading to Hom-Lie algebras using commutator map as new product, and in this sense constituting a natural generalization of associative algebras as Lie admissible algebras leading to Lie algebras using commutator map. Furthermore, in \cite{ms:homstructure}, more general $G$-Hom-associative algebras including Hom-associative algebras, Hom-Vinberg algebras (Hom-left symmetric algebras), Hom-pre-Lie algebras (Hom-right symmetric algebras), and some other Hom-algebra structures, generalizing $G$-associative algebras, Vinberg and pre-Lie algebras respectively, have been introduced and shown to be Hom-Lie admissible, meaning that for these classes of Hom-algebras, the operation of taking commutator leads to Hom-Lie algebras as well. Diagram \eqref{diagramhomliealg} illustrates the relations existing between some of these
structures. (Note that Rota-Baxter operators can be replaced by the more general
$\mathcal{O}$-operators in the upper and lower rows.)
 \begin{equation}\label{diagramhomliealg}
\begin{split}
\resizebox{13cm}{!}{\xymatrix{
\ar[rr] \mbox{\bf Hom-dendriform alg $(A,\prec,\succ,\alpha)$ }\ar[d]_{\mbox{ $\ast=\prec+\succ$}}\ar[rr]^{\mbox{\quad\quad $\cdot=\prec-\succ$\quad}}
                && \mbox{\bf Hom-pre-Lie alg $(A,\cdot ,\alpha)$ }\ar[d]_{\mbox{ Commutator}}\\
\ar[rr] \mbox{\bf Hom-associative alg $(A,\ast,\alpha)$}\ar@<-1ex>[u]_{\mbox{ R-B }}\ar[rr]^{\mbox{\quad\quad Commutator\quad}}
                && \mbox{\bf Hom-Lie alg  $(A,[-,-],\alpha)$}\ar@<-1ex>[u]_{\mbox{ R-B}}}
}\end{split}
\end{equation}
Also, flexible Hom-algebras have been introduced, connections to Hom-algebra generalizations of derivations and of adjoint maps have been noticed, and some low-dimensional Hom-Lie algebras have been described.
In Hom-algebra structures, defining algebra identities are twisted by linear maps.
Since the pioneering works
\cite{HartwigLarSil:defLiesigmaderiv,LarssonSilvJA2005:QuasiHomLieCentExt2cocyid,LarssonSilv:quasiLiealg,LSGradedquasiLiealg,LarssonSilv:quasidefsl2,ms:homstructure}, Hom-algebra structures have developed in a popular broad area with increasing number of publications in various directions.
Hom-algebra structures include their classical counterparts and open new broad possibilities for deformations, extensions to Hom-algebra structures of representations, homology, cohomology and formal deformations, Hom-modules and hom-bimodules, Hom-Lie admissible Hom-coalgebras, Hom-coalgebras, Hom-bialgebras, Hom-Hopf algebras, $L$-modules, $L$-comodules and Hom-Lie quasi-bialgebras, $n$-ary generalizations of biHom-Lie algebras and biHom-associative algebras, generalized derivations, Rota-Baxter operators, Hom-dendriform color algebras, Rota-Baxter bisystems and covariant bialgebras, Rota-Baxter cosystems, coquasitriangular mixed bialgebras, coassociative Yang-Baxter pairs, coassociative Yang-Baxter equation and generalizations of Rota-Baxter systems and algebras, curved $\mathcal{O}$-operator systems and their connections with tridendriform systems and pre-Lie algebras, BiHom-algebras, BiHom-Frobenius algebras and double constructions, infinitesimal biHom-bialgebras and Hom-dendriform $D$-bialgebras, and category theory of Hom-algebras
\cite{AmmarEjbehiMakhlouf:homdeformation,
AttanLaraiedh:2020ConstrBihomalternBihomJordan,
Bakayoko:LaplacehomLiequasibialg,
Bakayoko:LmodcomodhomLiequasibialg,
BakBan:bimodrotbaxt,
BakyokoSilvestrov:HomleftsymHomdendicolorYauTwi,
BakyokoSilvestrov:MultiplicnHomLiecoloralg,
BenMakh:Hombiliform,
BenAbdeljElhamdKaygorMakhl201920GenDernBiHomLiealg,
CaenGoyv:MonHomHopf,
DassoundoSilvestrov2021:NearlyHomass,
GrMakMenPan:Bihom1,
HassanzadehShapiroSutlu:CyclichomolHomasal,
HounkonnouDassoundo:centersymalgbialg,
HounkonnouDassoundo:homcensymalgbialg,
HounkonnouHoundedjiSilvestrov:DoubleconstrbiHomFrobalg,
kms:narygenBiHomLieBiHomassalgebras2020,
Laraiedh1:2021:BimodmtchdprsBihomprepois,
LarssonSilvJA2005:QuasiHomLieCentExt2cocyid,
LarssonSilv:quasidefsl2,
LarssonSigSilvJGLTA2008:QuasiLiedefFttN,
LarssonSilvestrovGLTMPBSpr2009:GenNComplTwistDer,
MaMakhSil:CurvedOoperatorSyst,
MaMakhSil:RotaBaxbisyscovbialg,
MaMakhSil:RotaBaxCosyCoquasitriMixBial,
MaZheng:RotaBaxtMonoidalHomAlg,
MabroukNcibSilvestrov2020:GenDerRotaBaxterOpsnaryHomNambuSuperalgs,
Makhlouf2010:ParadigmnonassHomalgHomsuper,
MakhSil:HomHopf,
MakhSilv:HomDeform,
MakhSilv:HomAlgHomCoalg,
MakYau:RotaBaxterHomLieadmis,
RichardSilvestrovJA2008,
RichardSilvestrovGLTbnd2009,
SaadaouSilvestrov:lmgderivationsBiHomLiealgebras,
ShengBai:homLiebialg,
Sheng:homrep,
SilvestrovParadigmQLieQhomLie2007,
SigSilv:GLTbdSpringer2009,
SilvestrovZardeh2021:HNNextinvolmultHomLiealg,
Yau:ModuleHomalg,
Yau:HomEnv,
Yau:HomHom,
Yau:HombialgcomoduleHomalg,
Yau:HomYangBaHomLiequasitribial}.

A Hom-type generalization of Malcev algebras, called Hom-Malcev
algebras, is defined in \cite{Yau}, where connections between Hom-alternative algebras and Hom-Malcev algebras are given. We aim in this paper to introduce and study, through Rota-Baxter operators and $\mathcal{O}$-operators, the relationship between Hom-Malcev, Hom-pre-Malcev algebras and Hom-M-dendriform algebras generalizing, then, Malcev  algebras, pre-Malcev algebras and M-dendriform algebras.
The anti-commutator of a Hom-pre-Malcev algebra is a Hom-Malcev algebra and the left multiplication operators give a representation of this Hom-Malcev algebra, which is the beauty of such a structure.
Similarly, a Hom-M-dendriform algebra gives rise to a Hom-pre-Malcev algebra, and a Hom-Malcev algebra, in the same way as Hom-L-dendriform algebra, gives rise to Hom-pre-Lie algebra and Hom-Lie algebra (see \cite{ChtiouiMabroukMakhlouf1,ChtiouiMabroukMakhlouf2,MakhloufHomdemdoformRotaBaxterHomalg2011} for more details).

In this paper we use $\mathcal{O}$-operators to split operations, although a generalization exists in the alternative setting in terms of bimodules: the Rota-Baxter operators
defined by Madariaga \cite{Madariaga}. Diagram \eqref{diagramhommalcev} summarizes the results of the present work.

In Section \ref{sec: prelimenaires}, we  summarize the definitions and some key constructions of  Hom-alternative algebras and Hom-Malcev algebras, we introduce the notion of $\mathcal{O}$-operator of Hom-Malcev algebras that generalizes the notion of Rota-Baxter operators and we develop (dual) representation of a Hom-Malcev algebras. In Section \ref{sec:Hom-pre-Malcev  algebras}, we introduce the notion of Hom-pre-Malcev algebra, provide  some properties and define the notion of a bimodule of a Hom-pre-Malcev algebra. Moreover, we develop  some constructions theorems. We show that on one hand, an $\mathcal{O}$-operator on a Hom-Malcev algebra gives a Hom-pre-Malcev algebra. On the other hand, a Hom-pre-Malcev algebra naturally gives an
$\mathcal{O}$-operator on the sub-adjacent Hom-Malcev algebra.
In Section \ref{sec:Hom-M-dendriform algebras}, we introduce the
 notion of Hom-M-dendriform algebra and then study some of their fundamental properties  in terms of the $\mathcal{O}$-operators of Hom-pre-Malcev algebras.  Their relationship with
Hom-Malcev algebras and Hom-alternative quadri-algebras are also described.
 Section \ref{sec:twistings} is devoted to showing that bimodules over Malcev and pre-Malcev
are twisted into bimodules over Hom-Malcev and Hom-pre-Malcev via endomorphisms
respectively.

\section{Preliminaries and basics}\label{sec: prelimenaires}
 The purpose of this section is to recall some basics about Hom-alternative  algebras introduced in \cite{Makhl:HomaltHomJord, Yau}. Moreover, we  give the definition of Hom-Malcev  algebras which may be viewed as a Hom-alternative algebra via the commutator bracket (see \cite{Yau}).


In this paper, all vector spaces are over a field $\mathbb{K}$ of characteristic $0$.

A Hom-algebra is a triple $(A, \mu, \alpha)$ in which $A$ is a vector space, $\mu: A^{\otimes2}\longrightarrow A$ is a bilinear map and $\alpha: A\longrightarrow  A$ is a linear map
(the twisting map). Hom-algebra is said to be multiplicative if $\alpha\circ \mu = \mu \circ \alpha^{\otimes 2}.$ Since many of the results in this paper depend on this property, we will assume the multiplicativity property for Hom-algebras as default in the paper, and thus will not write the word multiplicative each time, for simplicity of exposition. Also, when there is no ambiguity, we denote for simplicity the multiplication and composition by concatenation.

 \begin{defn}[\cite{Makhl:HomaltHomJord}]
A Hom-alternative algebra is a Hom-algebra $(A, \ast, \alpha)$ satisfying
for all $x, y, z \in A$,
\begin{equation}
\label{homalt}
as_{\alpha}(x,x,y)= as_{\alpha}(y, x, x)= 0,
\end{equation}
where $as_{\alpha}(x, y, z) = (x\ast y)\ast\alpha(z) -\alpha(x)\ast (y\ast z)$ is  the Hom-associator \textup{(}$\alpha$-associator\textup{)}.
\end{defn}

 \begin{defn}[\cite{Yau}]
A Hom-Malcev algebra is a Hom-algebra $(A, [-, -], \alpha)$ such that $[-, -]$ is anti-symmetric, and satisfies the Hom-Malcev identity for all $x, y, z \in A$,
\begin{equation}
\label{Hom-Malcev:Jacobiannotation}
J_{\alpha}(\alpha (x), \alpha (y), [x,z]) = [J_{\alpha}(x, y, z), \alpha^{2}(x)],
\end{equation}
where $J_{\alpha}(x, y, z)=[[x,y],\alpha (z)]+[[y,z],\alpha (x)]+[[z,x],\alpha (y)]$ is the Hom-Jacobian of $x, y, z$ \textup{(}$\alpha$-Jacobian\textup{)}.
\end{defn}

The Hom-Malcev identity \eqref{Hom-Malcev:Jacobiannotation} is equivalent to
\begin{align}\label{Hom-Malcev}
[\alpha ([x, z]), \alpha ([y, t])] &= [[[x, y], \alpha (z)], \alpha^{2}(t)] + [[[y, z], \alpha (t)], \alpha^{2} (x)]\\&  \nonumber
 + [[[z, t], \alpha (x)], \alpha^{2} (y)] + [[[t, x], \alpha (y)], \alpha^{2} (z)].
\end{align}
When $\alpha = Id_A$, we recover the Malcev algebra (see \cite{Malcev55:anltcloops}).

If the map $\alpha$ satisfies
$\alpha([x, y]) = [\alpha(x), \alpha(y)],$
then the Hom-Malcev algebra is said to be {\it multiplicative}.
Throughout this article we will impose multiplicative property on $\alpha$ because many of our results depend on it, and thus we will not write the word multiplicative each time.

Hom-alternative algebras related to Hom-Malcev algebras via admissibility \cite{Yau}.
\begin{thm} Any Hom-alternative algebra $(A, \ast, \alpha)$ is a Hom-Malcev admissible algebra. That is $(A, [- , -], \alpha)$ is a Hom-Malcev algebra with the bilinear bracket
defined by commutator $[x, y] = x\ast y - y\ast x$ for all $x, y \in A$.
\end{thm}
 Let  $(A, [-, -], \alpha)$ and  $(A', [-, -]', \alpha')$ be two Hom-Malcev algebras. A
linear map $f : A\to A' $ is said to be a morphism of Hom-Malcev algebras
if for all $ x, y\in A,$
$$[f(x)\, f(y)]' = f([x, y])~~\text{and}~~ f \circ \alpha = \alpha'\circ f$$
\begin{defn}
Let $(A,[-, -], \alpha)$ be a Hom-Malcev algebra and $V$ be a vector space, a linear map $\rho : A \longrightarrow End(V )$ is a  representation  of $(A,[-, -], \alpha)$ on $V$ with respect to $\beta \in
End(V )$ if for any $x, y, z \in A,$
\begin{eqnarray} \label{rephommalcev}
\rho(\alpha(x))\beta &=&  \beta\rho(x),\\
\rho([[x, y], \alpha (z)])\beta^{2} & =& \rho(\alpha^{2}(x))\rho(\alpha (y))\rho(z)-  \rho(\alpha^{2}(z))\rho(\alpha (x))\rho(y) \nonumber \\
& + & \rho(\alpha^{2}(y))\rho([z,x]) \beta -\rho(\alpha ([y, z]))\rho(\alpha (x))\beta.
\label{representation H-M}
\end{eqnarray}

\end{defn}
\begin{prop} \label{semidirectprduct HomMalcev}
Let $(A,[-, -], \alpha)$ be a Hom-Malcev algebra, $(V, \beta)$ a vector space and $\rho : A \longrightarrow End(V )$ a linear map. Then $(V, \rho, \beta)$ is a representation of $A$ if and only if
$(A \oplus V, [-, -]_{\rho}, \alpha + \beta)$ is a Hom-Malcev algebra, where $[-, -]_{\rho}$ and $\alpha + \beta$ are defined for all  $x, y \in  A,\ a, b \in V$ by
\begin{eqnarray*}
[x + a, y + b]_{\rho} &=& [x, y] + \rho(x)a - \rho(y)b, \\
(\alpha + \beta)(x + a) &=& \alpha(x) + \beta(a).
\end{eqnarray*}
This Hom-Malcev algebra is called the semi-direct product of $(A,[-, -], \alpha)$ and $(V, \beta)$, and denoted by
$A \ltimes _{\rho}^{\alpha, \beta} V$ or simply $A \ltimes V$.
\end{prop}
\begin{ex}
 Let $(A,[-, -], \alpha)$ be a Hom-Malcev algebra.
Then $ad_{x} : A \longrightarrow End(A)$ definedfor all $x, y \in A$ by
$ad_{x}(y) = [x , y],$ is a representation of $(A, [-, -], \alpha)$  on $A$  with respect to $\alpha$, called the adjoint representation of $A$.
\end{ex}
Let $(V,\rho,\beta)$ be a representation of a Hom-Malcev algebra $(A,[-,-],\alpha)$. In the sequel, we always assume that $\beta$ is invertible. Define $\rho^*:A\longrightarrow gl(V^*)$ as usual by
$$\langle \rho^*(x)(\xi),a\rangle=-\langle\xi,\rho(x)(a)\rangle,\quad\forall\ x\in A,\ a\in V,\ \xi\in V^*.$$
However, in general $\rho^*$ is not a representation of $A$. Define $\rho^\star:A\longrightarrow gl(V^*)$ by
\begin{equation}\label{eq:new1}
 \rho^\star(x)(\xi)=\rho^*(\alpha(x))\big((\beta^{-2})^*(\xi)\big),\quad\forall\ x\in A, \ \xi\in V^*.
\end{equation}
More precisely, we have
\begin{eqnarray}\label{eq:new1gen}
\langle\rho^\star(x)(\xi),a\rangle=-\langle\xi,\rho(\alpha^{-1}(x))(\beta^{-2}(a))\rangle,\quad\forall\ x\in A,\ a\in V,\ \xi\in V^*.
\end{eqnarray}
\begin{lem}\label{lem:dualrep}
 Let $(V,\rho,\beta)$ be a representation of a Hom-Malcev algebra $(A,[-,-],\alpha)$. Then $\rho^\star:A\longrightarrow gl(V^*)$ defined above by \eqref{eq:new1} is a representation of $(A,[-,-],\alpha)$ on $V^*$ with respect to $(\beta^{-1})^*$, which is called the  dual representation of $(V,\rho,\beta)$.
\end{lem}
\begin{proof}
For all $x\in A$, $\xi\in V^*$,  we have
\begin{align*}
\rho^\star(\phi(x))((\beta^{-1})^*(\xi)) &= \rho^*(\phi^{2}(x))(\beta^{-3})^*(\xi) \\
&=(\beta^{-1})^*(\rho^*(\phi(x))(\beta^{-2})^*(\xi))=(\beta^{-1})^*(\rho^\star(x)(\xi)),
\end{align*}
which implies $\rho^\star\big{(}\phi(x)\big{)}\circ(\beta^{-1})^*=(\beta^{-1})^*\circ\rho^\star(x)$.
On the other hand, for all $x,y\in A$, $\xi\in V^*$ and $a\in V$, we have
\begin{align*}
&\langle \rho^{\star}([[x,y],\alpha(z)])((\beta^{-2})^*(\xi)),a\rangle
=\langle\rho^*(\alpha ([[x,y],\alpha(z)]))((\beta^{-4})^*(\xi)),a\rangle\\
& =-\langle ((\beta^{-4})^*(\xi)),\rho(\alpha([[x,y],\alpha(z)]))(a)\rangle\\
& =-\langle ((\beta^{-4})^*(\xi)),\rho(\alpha^{3}(x))\rho(\alpha^{2}(y))\rho(\alpha(z))\beta^{-2}(a) -\rho(\alpha^{3}(z))\rho(\alpha^{2}(x))\rho(\alpha(y))\beta^{-2}(a)\\
 &\quad + \rho(\alpha^{3}(y))\rho(\alpha([z, x]))\beta^{-1}(a) - \rho(\alpha^{2}([y, z]))\rho(\alpha^{2}(x))\beta^{-1}(a)\rangle\\
& =-\langle ((\beta^{-6})^*(\xi)),\rho(\alpha^{4}(x))\rho(\alpha^{3}(y))\rho(\alpha^{2}(z))(a) - \rho(\alpha^{4}(z))\rho(\alpha^{3}(x))\rho(\alpha^{2}(y))(a)\\
 &\quad + \rho(\alpha^{4}(y))\rho(\alpha^{2}([z, x]))\beta(a)
 - \rho(\alpha^{3}([y, z]))\rho(\alpha^{3}(x))\beta(a)\rangle\\
&=\langle\rho^{\ast}(\alpha^{4}(x))\rho^{\ast}(\alpha^{3}(y))\rho^{\ast}(\alpha^{2}(z))((\beta^{-6})^*(\xi)) -\rho^{\ast}(\alpha^{4}(z))\rho^{\ast}(\alpha^{3}(x))\rho^{\ast}(\alpha^{2}(y))((\beta^{-6})^*(\xi))\\
&\quad+ \rho^{\ast}(\alpha^{4}(y))\rho^{\ast}(\alpha^{2}([z, x]))\beta^{-5} - \rho^{\ast}(\alpha^{3}([y, z]))\rho^{\ast}(\alpha^{3}(x))((\beta^{-5})^*(\xi)), a\rangle\\
&=\langle\rho^{\star}(\alpha^{2}(x))\rho^{\star}(\alpha (y))\rho^{\star}(z)-  \rho^{\star}(\alpha^{2}(z))\rho^{\star}(\alpha (x))\rho^{\star}(y)\\  &\quad+\rho^{\star}(\alpha^{2}(y))\rho^{\star}([z,x]) \beta^{-1} -\rho^{\star}(\alpha ([y, z]))\rho^{\star}(\alpha (x))\beta^{-1},a\rangle,
\end{align*}
which implies that
\begin{align*}
\rho^\star([[x,y],\alpha(z)])\circ(\beta^{-2})^*&=\rho^{\star}(\alpha^{2}(x))\rho^{\star}(\alpha (y))\rho^{\star}(z)-  \rho^{\star}(\alpha^{2}(z))\rho^{\star}(\alpha (x))\rho^{\star}(y)\\
&+\rho^{\star}(\alpha^{2}(y))\rho^{\star}([z,x]) \beta^{-1} -\rho^{\star}(\alpha ([y, z]))\rho^{\star}(\alpha (x))\beta^{-1}.
\end{align*}
Therefore, $\rho^\star$ is a representation of $(A,[-,-],\alpha)$ on $V^*$ with respect to $(\beta^{-1})^*$.
\end{proof}
\begin{lem}\label{lem:dualdual}
If $(V,\rho,\beta)$ is a representation of Hom-Malcev algebra  $(A,[-,-],\alpha)$, then  $$(\rho^\star)^\star=\rho.$$
\end{lem}
\begin{proof}
For all $x\in A,a\in V,\xi\in V^*$, we have
\begin{eqnarray*}
\langle (\rho^\star)^\star(x)(a),\xi\rangle
&=&\langle(\rho^\star)^*(\alpha(x))(\beta^2(a)),\xi\rangle
=-\langle\beta^2(a),\rho^\star(\alpha(x))(\xi)\rangle\\
&=&-\langle\beta^2(a),\rho^*(\alpha^2(x))((\beta^{-2})^*(\xi))\rangle
=\langle\rho(\alpha^2(x))(\beta^2(a)),(\beta^{-2})^*(\xi)\rangle\\
&=&\langle\rho(x)(a),\xi\rangle,
\end{eqnarray*}
which implies that $(\rho^\star)^\star=\rho$.
\end{proof}

\begin{cor}\label{lem:rep}
 If $(A,[-,-],\alpha)$ is a Hom-Malcev algebra, then $ad^{\star}:A\rightarrow gl(A^*)$  defined
 for all $x\in A, \ \xi\in A^*$ by
 \begin{equation}
  ad^\star(x)\xi=ad^*(\alpha(x))(\alpha^{-2})^*(\xi),
 \end{equation}
 is a representation of the Hom-Malcev algebra $(A,[-,-],\alpha)$ on $A^*$ with respect to $(\alpha^{-1})^*$, called the coadjoint representation.
\end{cor}

The following terminology is motivated by the notion of $\mathcal{O}$-operator as a generalization of Rota-Baxter operator of weight $0$.
\begin{defn}
A linear map $T : V \to  A$ is called an $\mathcal{O}$-operator  associated to a representation $(V,\rho,\beta)$ of a Hom-Malcev algebra $(A, [-, -], \alpha)$ if
for all $a, b \in V,$
\begin{align}\label{ophommalcev}
 \alpha\circ T =  T\circ\beta, \quad \quad [T (a), T (b)] = T \big(\rho(T (a))b - \rho(T (b))a\big).
 \end{align}
\end{defn}
 \begin{ex}
 A Rota-Baxter operator of  weight $0$ on a Hom-Malcev algebra $A$ is
just an $\mathcal{O}$-operator associated to the adjoint representation $(A,ad,\alpha)$, that is, $\mathcal{R}$ satisfies
$$ \mathcal{R}\circ\alpha = \alpha\circ \mathcal{R}, \quad \quad [\mathcal{R}(x),\mathcal{R}(y)] = \mathcal{R}([\mathcal{R}(x), y] + [x, \mathcal{R}(y)]),$$
for all $x,y\in A$.
\end{ex}

 \section{Hom-pre-Malcev  algebras}\label{sec:Hom-pre-Malcev  algebras}
In this section, we generalize the notion of pre-Malcev algebras introduced in \cite{Madariaga} to the Hom
case and study the relationships with Hom-Malcev algebras in terms of $\mathcal{O}$-operators of Hom-Malcev algebras and Hom-pre-alternative algebras. Moreover we characterize the representation of Hom-pre-Malcev algebras and provide some key constructions.
\subsection{Definition and basic properties}\label{subsec:Definition and basic properties}
\begin{defn}
A  Hom-pre-Malcev algebra is  a  Hom-algebra $(A, \cdot, \alpha)$ satisfying, for any $x,y,z,t \in A$ and $[x,y]=x\cdot y-y\cdot x,$ the identity
\begin{equation}\label{HPM}
\begin{split}
  &[\alpha(y),\alpha(z)] \cdot \alpha(x \cdot t) + [[x , y] , \alpha(z)] \cdot \alpha^{2}(t) \\
&\qquad + \alpha^{2}(y) \cdot ([x , z] \cdot \alpha (t)) - \alpha^{2}(x) \cdot (\alpha (y)  \cdot (z \cdot t)) + \alpha^{2}(z) \cdot (\alpha (x) \cdot (y \cdot t))=0.\end{split}
\end{equation}
\end{defn}

The identity \eqref{HPM} is equivalent to
$HPM(x, y, z, t) = 0$, where for all $x,y,z,t\in A,$
\begin{equation}\label{HPMexpanded}
\begin{split}
HPM(x, y, z, t) &=  \alpha(y \cdot z) \cdot \alpha(x \cdot t) - \alpha(z \cdot y) \cdot \alpha(x \cdot t)\\
&\quad+ ((x \cdot y) \cdot \alpha(z)) \cdot \alpha^{2}(t) - ((y \cdot x) \cdot \alpha (z)) \cdot \alpha^{2}(t)\\
 &\quad- (\alpha (z) \cdot (x \cdot y)) \cdot \alpha^{2}(t) + (\alpha (z) \cdot (y \cdot x)) \cdot \alpha^{2}(t)\\
&\quad+ \alpha^{2}(y) \cdot ((x \cdot z) \cdot \alpha (t)) - \alpha^{2}(y) \cdot ((z \cdot x) \cdot \alpha (t))\\
 &\quad- \alpha^{2}(x) \cdot (\alpha (y)  \cdot (z \cdot t)) + \alpha^{2}(z) \cdot (\alpha (x) \cdot (y \cdot t)).
\end{split}
\end{equation}
A Hom-pre-Malcev algebra is said to be a multiplicative Hom-pre-Malcev algebra if $\alpha$ satisfies
$\alpha(x \cdot y) = \alpha(x) \cdot \alpha(y),$ for all $x, y \in A.$

Hom-pre-Malcev algebras generalize Hom-pre-Lie algebras. A Hom-pre-Lie algebra
 is a vector space $A$ with a bilinear product $\cdot$ and a linear map $\alpha$ satisfying the Hom-pre-Lie
identity for all $x, y, z \in A$,
$$HPL(x, y, z) = as_\alpha(x, y, z) - as_\alpha(y, x, z) = 0,$$
where $as_\alpha(x, y, z)=(x\cdot y)\cdot\alpha(z)-\alpha(x)\cdot(y\cdot z)$ is the Hom-associator. Note that
\begin{align*}
HPM(x, y, z, t) &= HPL([x, y],\alpha (z), \alpha (t)) - HPL([y, x], \alpha (z),  \alpha (t)) \\
& \quad+ HPL(\alpha (x), \alpha (y), [z, t])+ HPL(\alpha (y), \alpha (z), [x, t]) \\
& \quad- [\alpha^{2}(z), HPL(x, y, t)] + [\alpha^{2}(y), HPL(x, z, t)].
\end{align*}
So every Hom-pre-Lie algebra is a Hom-pre-Malcev algebra.

When $\alpha = Id_A$, Hom-pre-Malcev algebra $(A, \cdot, \alpha)$ is a pre-Malcev algebra.

\begin{defn}
 Let  $(A, \cdot, \alpha)$ and  $(A', \cdot', \alpha')$ be two Hom-pre-Malcev algebras. A
linear map $f : A\to A' $ is called a morphism of Hom-pre-Malcev algebras
if, for all $x, y\in A$,
$$f(x)\cdot' f(y) = f(x\cdot y),~~~~ f \circ \alpha = \alpha'\circ f. $$
\end{defn}

\begin{prop} \label{prop:HompreMalcevHomMalcevadmis}
 Let $(A, \cdot, \alpha)$ be a Hom-pre-Malcev algebra.
 The commutator given, for all $x,\ y\in A$, by
 \begin{align}\label{commutator}
 [x, y] = x \cdot y - y \cdot x,
 \end{align}
  defines a Hom-Malcev algebra structure on $A$.
 \end{prop}
\begin{proof}
We show that the commutator \eqref{commutator} satisfies
the identity \eqref{Hom-Malcev}. For $x, y, z, t \in A$,
\begin{align*}
&[\alpha ([x, z]), \alpha([y, t])] - [[[x, y], \alpha (z)], \alpha^{2} (t)] - [[[y, z], \alpha(t)], \alpha^{2}(x)] \\
&\quad - [[[z, t], \alpha (x)], \alpha^{2}(y)]- [[[t, x], \alpha (y)], \alpha^{2}(z)]\\
&=\alpha(x \cdot z) \cdot \alpha(y \cdot t) - \alpha (x \cdot z) \cdot \alpha(t \cdot y) - \alpha(z \cdot x) \cdot \alpha(y \cdot t) + \alpha(z  \cdot x) \cdot \alpha(t \cdot y)\\
&\quad - \alpha(y \cdot t) \cdot \alpha(x \cdot z)  + \alpha(y \cdot t) \cdot \alpha(z \cdot x)+ \alpha(t \cdot y) \cdot \alpha(x \cdot z) - \alpha(t \cdot y) \cdot \alpha(z \cdot x)\\
&\quad-  ((x \cdot y) \cdot \alpha(z)) \cdot \alpha^{2}(t) + ((y \cdot x) \cdot \alpha(z)) \cdot \alpha^{2}(t) + (\alpha(z) \cdot (x \cdot y)) \cdot \alpha^{2}(t)\\
&\quad- (\alpha(z) \cdot (y \cdot x)) \cdot \alpha^{2}(t)+ \alpha^{2}(t) \cdot ((x \cdot y) \cdot \alpha (z)) - \alpha^{2}(t) \cdot ((y \cdot x) \cdot \alpha (z))\\
&\quad- \alpha^{2}(t) \cdot (\alpha (z) \cdot (x \cdot y))+ \alpha^{2}(t) \cdot (\alpha(z) \cdot (y \cdot x))- ((y \cdot z) \cdot \alpha (t)) \cdot \alpha^{2}(x)\\
&\quad+ ((z \cdot y) \cdot \alpha(t)) \cdot \alpha^{2}(x)+ (\alpha(t) \cdot (y \cdot z))  \cdot \alpha^{2}(x) - (\alpha(t) \cdot (z \cdot y)) \cdot \alpha^{2}(x)\\
&\quad+ \alpha^{2}(x) \cdot ((y \cdot z) \cdot \alpha(t))- \alpha^{2}(x) \cdot ((z \cdot y) \cdot \alpha(t)) - \alpha^{2}(x) \cdot (\alpha(t) \cdot (y \cdot z))\\
 &\quad+ \alpha^{2}(x) \cdot (\alpha(t) \cdot (z \cdot y))- ((z \cdot t) \cdot \alpha(x)) \cdot \alpha^{2}(y) + ((t \cdot z) \cdot \alpha(x)) \cdot \alpha^{2}(y)\\
 &\quad+ (\alpha(x) \cdot (z \cdot t)) \cdot \alpha^{2}(y)- (\alpha(x) \cdot (t \cdot z)) \cdot \alpha^{2}(y)+ \alpha^{2}(y) \cdot ((z \cdot t) \cdot \alpha(x))\\
 &\quad- \alpha^{2}(y) \cdot ((t \cdot z) \cdot \alpha(x))- \alpha^{2}(y) \cdot (\alpha(x) \cdot (z \cdot t)) + \alpha^{2}(y) \cdot (\alpha(x) \cdot (t \cdot z))\\
&\quad- ((t \cdot x) \cdot \alpha(y)) \cdot \alpha^{2}(z)+ ((x \cdot t) \cdot \alpha(y)) \cdot \alpha^{2}(z)+ (\alpha(y) \cdot (t \cdot x)) \cdot \alpha^{2}(z)\\
&\quad+ (\alpha(y) \cdot (x \cdot t)) \cdot \alpha^{2}(z)+ \alpha^{2}(z) \cdot ((t \cdot x) \cdot \alpha(y)) - \alpha^{2}(z) \cdot ((x \cdot t) \cdot \alpha(y))\\
 &\quad- \alpha^{2}(z) \cdot (\alpha(y) \cdot (t \cdot x))+ \alpha^{2}(z) \cdot (\alpha(y) \cdot (x \cdot t))\\
&= HPM(x, t, y, z)  + HPM(y, x, z, t) + HPM(z, y, t, x)  + HPM(t, z, x, y) = 0.
\qedhere   \end{align*}
 \end{proof}

\begin{defn}
The Hom-Malcev algebra structure in Proposition \ref{prop:HompreMalcevHomMalcevadmis} is called the associated Hom-Malcev algebra of the Hom-pre-Malcev algebra $(A,\cdot, \alpha)$, and the  Hom-pre-Malcev algebra $(A, \cdot, \alpha)$ is called a compatible Hom-pre-Malcev algebra on the Hom-Malcev algebra $(A, [-, -], \alpha)$.
\end{defn}

A Hom-pre-Malcev algebra can be viewed as a Hom-Malcev algebra whose
operation decomposes into two compatible pieces.

Examples of Hom-pre-Malcev algebras can be constructed from Hom-Malcev algebras with $\mathcal{O}$-operators. Let $(A,[-,-],\alpha)$ be a Hom-Malcev algebra  and $T : V \to A$ be an $\mathcal{O}$-operator of  $A$
associated to a module $(V,\rho, \beta)$. Define the product $"\cdot"$ by \begin{equation}
a\cdot b= \rho(T(a))b, \quad \forall\ a, b \in V.
\end{equation}

\begin{prop}\label{hommalcev==>hompremalcev}
With the above notations $(V,\cdot, \beta)$ is a Hom-pre-Malcev algebra, and there exists an associated Hom-Malcev algebra structure on $V$ given by \eqref{commutator} and $T$
is a homomorphism of Hom-Malcev algebras. Furthermore, $T(V) = \{T(a),\ a \in V \}\subset A$ is a Hom-Malcev
subalgebra of $(A, [-, -], \alpha)$ and $(T (V ),[-, -],
\alpha)$ is a Hom-pre-Malcev algebra structure given,
for all $a, b \in V,$ by $T(a)\cdot T(b) = T(a\cdot b). $
Moreover, the corresponding associated Hom-Malcev algebra structure on  $T (V)$ given by \eqref{commutator} is just a
Hom-Malcev subalgebra  structure of $(A, [-, -], \alpha)$, and $T$ is a morphism of Hom-pre-Malcev algebras.
\end{prop}
\begin{proof}
By the identity of $\mathcal{O}$-operator \eqref{ophommalcev},
$T([a,b]_{T})=T(a\cdot b-b\cdot a)=[T(a), T(b)].$
Thanks to \eqref{representation H-M}, for any $a,b,c,d\in V$,
\begin{align*}
&[\beta(b),\beta(c)]_{T } \cdot \beta(a \cdot d) + [[a ,b]_{T } , \beta(c)]_{T } \cdot \beta^{2}(d) + \beta^{2}(b) \cdot ([a , c]_{T} \cdot \beta (d))\\
 &\quad- \beta^{2}(a) \cdot (\beta (b)  \cdot (c \cdot d)) + \beta^{2}(c) \cdot (\beta (a) \cdot (b \cdot d))\\
  &= \rho(\alpha(T([b,c]_{T}))\rho(\alpha(T(a)))\beta(d)+\rho(T([[a, b]_{T},\beta(c)]_{T}))\beta^{2}(d)\\
  &\quad+\rho(\alpha^{2}(T(b)))\rho(T([a,c]_{T}))\beta(d)-  \rho(\alpha^{2}(T(a)))\rho(\alpha(T(b)))\rho(T(c))d\\
& \quad+\rho(\alpha^{2}(T(c)))\rho(\alpha(T(a)))\rho(T(b))d\\
 &= \rho(\alpha([T(b), T(c)]))\rho(\alpha(T(a)))\beta(d)+\rho([[T(a), T(b)], \alpha(T(c))])\beta^{2}(d)\\
  &\quad+\rho(\alpha^{2}(T(b)))\rho([T(a), T(c)])\beta(d)-\rho(\alpha^{2}(T(a)))\rho(\alpha(T(b)))\rho(T(c))d\\
  &\quad+\rho(\alpha^{2}(T(c)))\rho(\alpha(T(a)))\rho(T(b))d= 0.
\end{align*}
So, $(V,\cdot,\beta)$ is a Hom-pre-Malcev algebra. The other statements follow immediately.
\end{proof}
An obvious consequence of Proposition \ref{hommalcev==>hompremalcev} is the construction of a Hom-pre-Malcev algebra in terms of a Rota-Baxter operator of weight zero of a Hom-Malcev
algebra.
\begin{cor}
Let $\mathcal{R}: A \longrightarrow A$ is a Rota-Baxter operator on a Hom-Malcev algebra
$(A,[-,-],\alpha)$.  Then $(A, \cdot, \alpha)$ is a Hom-pre-Malcev algebra, where for all $x, y \in A,$
$$x \cdot y = [\mathcal{R}(x), y].$$
\end{cor}
\begin{ex}
\label{ex:4dmalcev}
There is a four-dimensional Malcev algebra $(A,[-,-])$  with multiplication table {\rm \cite[Example 3.1]{Sagle}} for a basis $\{e_1,e_2,e_3,e_4\}$,
\begin{center}
\begin{tabular}{c|cccc}
$[-,-]$ & $e_1$ & $e_2$ & $e_3$ & $e_4$ \\ \hline
$e_1$ & $0$ & $-e_2$ & $-e_3$ & $e_4$ \\
$e_2$ & $e_2$ & $0$ & $2e_4$ & $0$ \\
$e_3$ & $e_3$ & $-2e_4$ & $0$ & $0$ \\
$e_4$ & $-e_4$ & $0$ & $0$ & $0$ \\
\end{tabular}
\end{center}
Let $\mathcal{R}$ be the operator defined, with respect to the basis $\{e_1,e_2,e_3, e_{4}\}$, by
\begin{align*}
    \mathcal{R}(e_1)= e_1 + \frac{a_{4}}{2}e_4,\ \mathcal{R}(e_2)=\lambda_1 e_3,\ \mathcal{R} (e_3)= \mathcal{R}(e_4) = 0,
\end{align*}
where $a_4$ and $\lambda_1$ are parameters in $\mathbb{K}$. By a direct computation, we can verify that $\mathcal{R}$ is a Rota-Baxter operator on $A$.

Now, using  the previous corollary,   there is a pre-Malcev algebra structure
on $A$ with the multiplication $"\cdot"$ given for all $x,y \in A$ by
$x\cdot y= [\mathcal{R}(x),  y]$, that is
$$
\begin{array}{c|cccc}
  \cdot  & e_1 & e_2 & e_3 &  e_4 \\   \hline
  e_1& \frac{-a_4}{2} e_4 & -e_2  & - e_3 & e_4 \\
  e_2 & \lambda_1 e_3 & -2\lambda_{1} e_4 &  0 & 0\\
  e_3 & 0 & 0  & 0 & 0\\
  e_4 & 0 & 0  & 0 & 0
  \end{array}$$
Using suitable algebra morphism $\alpha$, we can twist the Malcev algebra $A$ into Hom-Malcev algebras.
With a bit of computation, one can check that one class of algebra morphisms $\alpha \colon A \to A$ is given by
\[
\alpha(e_1) = e_1 + a_4e_4,\quad
\alpha(e_2) = -e_2 + b_3e_3,\quad
\alpha(e_3) = -e_3,\quad
\alpha(e_4) = -e_4,\quad
\]
where $a_4$ and $b_3$ are arbitrary scalars in $\mathbb{K}$.

There is a Hom-Malcev algebra $A_{\alpha} = (A,[-,-]_{\alpha} = \alpha\circ[-,-],\alpha)$
with multiplication table
\begin{center}
\begin{tabular}{c|cccc}
$[-,-]_{\alpha}$ & $e_1$ & $e_2$ & $e_3$ & $e_4$ \\ \hline
$e_1$ & $0$ & $-\alpha(e_2)$ & $e_3$ & $-e_4$ \\
$e_2$ & $\alpha(e_2)$ & $0$ & $ -2e_4$ & $0$ \\
$e_3$ & $-e_3$ & $2e_4$ & $0$ & $0$ \\
$e_4$ & $e_4$ & $0$ & $0$ & $0$ \\
\end{tabular}
\end{center}
Then we can check that $\mathcal{R}$ is a Rota-Baxter operator on Hom-Malcev algebra $A_{\alpha}$. Therefore, there exists a Hom-pre-Malcev algebra $ A_\alpha=(A,\cdot_{\alpha} ,\alpha)$  where  the  multiplication $"\cdot_{\alpha}"$ is  given by
$x\cdot_{\alpha} y = \alpha([\mathcal{R}(x), y])$ for all
$x,y \in A,$ that is
$$
\begin{array}{c|cccc}
  \cdot_{\alpha}  & e_1 & e_2 & e_3 &  e_4 \\ \hline
  e_1& \frac{a_4}{2} e_4 & -\alpha(e_2)  & e_3 & -e_4 \\
  e_2 & -\lambda_1 e_3 & 2 \lambda_1 e_4 &  0 & 0\\
  e_3 & 0 & 0  & 0 & 0\\
  e_4 &  0 & 0  & 0 & 0
  \end{array}$$
\end{ex}
\begin{ex}
\label{ex:5dmalcev}
There is a five-dimensional  Malcev algebra $(A,[-,-])$ with multiplication table {\rm \cite[Example 3.4]{Sagle}} for a basis $\{e_1,e_2,e_3,e_4,e_5\}$,
\begin{center}
\begin{tabular}{c|ccccc}
$[-,-]$ & $e_1$ & $e_2$ & $e_3$ & $e_4$ & $e_5$ \\ \hline
$e_1$ & $0$ & $0$ & $0$ & $e_2$ & $0$ \\
$e_2$ & $0$ & $0$ & $0$ & $0$ & $e_3$ \\
$e_3$ & $0$ & $0$ & $0$ & $0$ & $0$ \\
$e_4$ & $-e_2$ & $0$ & $0$ & $0$ & $0$ \\
$e_5$ & $0$ & $-e_3$ & $0$ & $0$ & $0$
\end{tabular}
\end{center}
Let $\mathcal{R}$ be the operator defined, with respect to the basis $\{e_1,e_2,e_3, e_{4}, e_{5}\}$, by
\begin{align*}
    \mathcal{R}(e_1)= e_1 + a_{4}e_4 + a_{5}e_{5},\ \mathcal{R}(e_2)=b e_3,\  \mathcal{R}(e_3)= \mathcal{R}(e_5) = 0, \mathcal{R}(e_4)= \frac{-b}{a_{5}}e_{2},
\end{align*}
where $b$, $a_4$ and $a_{5}$ are parameters in $\mathbb{K}$. By a direct computation, we can verify that $\mathcal{R}$ is a Rota-Baxter operator on $A$.

Now, using  the previous corollary,   there is a pre-Malcev algebra structure $"\cdot"$
on $A$  given for all $x,y \in A$ by
$x\cdot y= [\mathcal{R}(x),  y], $ that is
$$
\begin{array}{c|ccccc}
  \cdot  & e_1 & e_2 & e_3 &  e_4 & e_5\\  \hline
  e_1& -a_4 e_2 & -a_5e_3  & 0 & e_2 & 0 \\
  e_2 & 0  & 0 &  0 & 0 & 0\\
  e_3 & 0 & 0  & 0 & 0 & 0 \\
  e_4 & 0 & 0  & 0 & 0  & \frac{-b}{a_5}e_3\\
  e_5 & 0 & 0  & 0 & 0 & 0
  \end{array}$$
Using suitable algebra morphism $\alpha$ given by
\[
\alpha(e_1) = e_1 ,\quad
\alpha(e_2) = e_2, \quad
\alpha(e_3) = e_3, \quad
\alpha(e_4) = \lambda_2e_3 + e_4,\quad
\alpha(e_5) = \frac{a_4}{a_5}\lambda_2e_3 + e_5,
\]where $\lambda_2$, $a_4$ and $a_5$ are arbitrary scalars in $\mathbb{K}$, we can twist the Malcev algebra $A$ into Hom-Malcev algebra $A_{\alpha} = (A,[-,-]_{\alpha} = \alpha\circ[-,-],\alpha)$
with multiplication table
\begin{center}
\begin{tabular}{c|ccccc}
$[-,-]_{\alpha}$ & $e_1$ & $e_2$ & $e_3$ & $e_4$ &$e_5$ \\ \hline
$e_1$ & $0$ & $0$ & $0$ & $e_2$ & $0$\\
$e_2$ & $0$ & $0$ & $ 0$ & $0$ & $e_3$ \\
$e_3$ & $0$ & $0$ & $0$ & $0$ & $0$ \\
$e_4$ & $-e_2$ & $0$ & $0$ & $0$ & $0$\\
$e_5$ &$0$ & $-e_3$ & $0$ & $0$ & $0$
\end{tabular}
\end{center}
Then we can check that $\mathcal{R}$ is a Rota-Baxter operator on $A$. Thus there exists a Hom-pre-Malcev algebras $A_\alpha=(A,\cdot_{\alpha} ,\alpha)$ with a multiplication $"\cdot_{\alpha}"$ given by
$x\cdot_{\alpha} y = \alpha([\mathcal{R}(x), y])$ for all $x,y \in A,$
 that is
$$
\begin{array}{c|ccccc}
   \cdot_{\alpha}  & e_1 & e_2 & e_3 &  e_4 & e_5\\         \hline
  e_1& -a_4 e_2 & -a_5e_3  & 0 & e_2 & 0 \\
  e_2 & 0  & 0 &  0 & 0 & 0\\
  e_3 & 0 & 0  & 0 & 0 & 0 \\
  e_4 & 0 & 0  & 0  & 0 & \frac{-b}{a_5}e_3\\
 e_5 & 0 & 0  & 0 & 0 & 0
  \end{array}$$
\end{ex}
\begin{prop}\label{pro:nsc}
   Let $(A, [-,-],\alpha)$ be a Hom-Malcev algebra. Then there exists  a compatible Hom-pre-Malcev algebra structure on $A$ if and only if there is an invertible $\mathcal{O}$-operator on $A$.
\end{prop}
\begin{proof}
Let $(A,\cdot,\alpha)$ be a Hom-pre-Malcev algebra and $(A,[-,-],\alpha)$ be the associated Malcev algebra.  Then the identity map $id: A \to A$ is an invertible $\mathcal{O}$-operator on  $A$  associated to $(A,ad)$.

Conversely, suppose that there exists an invertible $\mathcal{O}$-operator $T$  of $(A,[-,-])$  associated to $(V, \rho,\beta)$. Then, using Proposition \ref{hommalcev==>hompremalcev}, there is a Hom-pre-Malcev algebra structure on $T(V)=A$ given for all $a, b\in V$ by
\begin{equation*}
   T(a)\cdot T(b)=T(\rho(T(a))b).
\end{equation*}
If we set $x=T(a)$ and $y=T(b)$, then we get
\begin{equation*}
   x\cdot y=T(\rho(x)T^{-1}(y)).
\end{equation*}
This is compatible Hom-pre-Malcev algebra structure  on $(A,[-,-],\alpha)$. Indeed,
\begin{align*}
 x\cdot y-y\cdot x&= T(\rho(x)T^{-1}(y) - \rho(y)T^{-1}(x) )\\
&= [TT^{-1}(x),TT^{-1}(y)]=[x,y].
\end{align*}
The proof is finished.
\end{proof}

 Let $(A,[-,-]\alpha)$ be an any Hom-algebra and   $\omega\in\wedge^2A^*$. Recall that $\omega$ is a symplectic structure on $A$ if it satisfies
\begin{align*}
    &\omega(\alpha(x),\alpha(y)) =\omega(x,y),\quad \quad  \displaystyle\circlearrowleft_{x,y,z}\omega([x,y],\alpha(z)) =0.
\end{align*}
A Hom-Malcev algebra
$(A,[-,-], \alpha)$ with a symplectic form is called a symplectic Hom-Malcev algebra.

\begin{cor}
Let $(A,[-,-],\alpha )$ be a  Hom-Malcev algebra and  $\omega$ be a symplectic structure on  $(A,[-,-],\alpha)$.
Then there is a compatible Hom-pre-Malcev algebra structure on $A$ given by
$$
 \omega(x\cdot y,\alpha(z))=\omega(\alpha(y),[z,x] ).
$$
\end{cor}

 \begin{proof}
   Since $(A,[-,-],\alpha )$ is a  Hom-Malcev algebra, $(A^*,ad^\star,(\alpha^{-1})^*)$ is a representation of $A$. By the fact that $\omega$ is a symplectic structure, $(\omega^\sharp)^{-1}$ is an $\mathcal{O}$-operator on the  Hom-Malcev algebra $(A,[-,-],\alpha )$ with respect to the representation $(A^*,ad^\star,(\alpha^{-1})^*)$. Thus, $(\omega^\sharp)^{-1}$ is an $\mathcal{O}$-operator on the   Hom-Malcev algebra $(A,[-,-] ,\alpha)$ with respect to the representation $(A^*,ad^\star,(\alpha^{-1})^*)$. By Proposition \ref{pro:nsc}, there is a compatible Hom-pre-Malcev algebra structure on $A$ given as above.
 \end{proof}
Hom-pre-Malcev algebras are related to Hom-pre-alternative algebras analogously
to how Hom-pre-Lie algebras are related to Hom-dendriform algebras \cite{MakhloufHomdemdoformRotaBaxterHomalg2011}.
\begin{defn}[\cite{Q.Sun}]
A Hom-pre-alternative algebra is a quadruple $(A,\prec,\succ,\alpha)$, where $A$ is a vector space,
$\prec,\succ: A \otimes A \longrightarrow A$
are bilinear maps and $\alpha\in gl(A)$  satisfying
for all $x, y, z \in A$ and  $x \ast y = x \prec y + x\succ y$,
\begin{align}
 &(x\succ y) \prec \alpha(z) - \alpha(x)\succ(y \prec z) + (y \prec x) \prec\alpha(z) - \alpha(y) \prec (x \ast z) = 0,\\
&(x\succ y)\prec\alpha(z) - \alpha(x)\succ(y \prec z) + (x \ast z)\succ \alpha(y) - \alpha(x)\succ(z\succ y) = 0,\\
&(x\ast y) \succ \alpha(z) - \alpha(x)\succ(y \succ z) + (y \ast x) \succ \alpha(z) - \alpha(y) \succ (x \succ z) = 0,\\
&(x\prec y)\prec\alpha(z) - \alpha(x)\prec(y \ast z) + (x \prec z)\prec \alpha(y) - \alpha(x)\prec(z\ast y) = 0.
\end{align}
\end{defn}
\begin{prop}
Let $(A,\prec,\succ, \alpha)$ be a Hom-pre-alternative algebra. Then $(A,\ast, \alpha)$ is a Hom-alternative algebra
\end{prop}
Now, we consider a Hom-alternative algebra $(A,\ast,\alpha)$, a vector space $V$ and a linear map $\beta:V\to V$. Recall that, a bimodule of $A$ with respect to $\beta$ is given by linear maps $\mathfrak{l},\mathfrak{r}:A\to End(V)$ satisfying the following conditions:
\begin{align}
\label{rephomalt1}\mathfrak{l}(x^2)\beta &=\mathfrak{l}(\alpha(x))\mathfrak{l}(x),\\
\label{rephomalt2}\mathfrak{r}(x^2)\beta &=\mathfrak{r}(\alpha(x))\mathfrak{r}(x),\\
\label{rephomalt3}\mathfrak{r}(\alpha(y))\mathfrak{l}(x)-\mathfrak{l}(\alpha(x))\mathfrak{r}(y)&=\mathfrak{r}(x\ast y)\beta-\mathfrak{r}(\alpha(y))\mathfrak{r}(x),\\
\label{rephomalt4} \mathfrak{l}(y\ast x)\beta-\mathfrak{l}(\alpha(y))\mathfrak{l}(x)&=\mathfrak{l}(\alpha(y))\mathfrak{r}(x)-\mathfrak{r}(\alpha(x))\mathfrak{l}(y).
\end{align}
\begin{defn}
An $\mathcal{O}$-operator of Hom-alternative algebra $(A,\ast,\alpha)$ with respect to the bimodule $(V,\mathfrak{l},\mathfrak{r},\beta)$ is a linear map $T:V\to A$ such that,
for all $a, b \in V$,
\begin{equation}
\label{O-ophomalternative} \alpha\circ T =  T\circ\beta ~~\text{and}~~T (a)\ast T (b) = T \big(\mathfrak{l}(T (a))b + \mathfrak{r}(T (b))a\big).
 \end{equation}
 \end{defn}
 \begin{rem}
Rota-Baxter operator of weight $0$ on a Hom-alternative algebra $(A,\ast, \alpha)$ is
just an $\mathcal{O}$-operator associated to the bimodule $(A,L, R,\alpha)$, where $L$ and $R$ are the left and right
multiplication operators corresponding to the multiplication $\ast$.
 \end{rem}
\begin{prop}
With the above notations, the triplet $(V,\mathfrak{l}-\mathfrak{r},\beta)$ defines a module of the Hom-Malcev admissible algebra $(A,[-,-],\alpha)$, and $T$ is an $\mathcal{O}$-operator of $(A,[-,-],\alpha)$ with respect to $(V,\mathfrak{l}-\mathfrak{r},\beta)$.
\end{prop}
\begin{proof}
First, note that  $(V,\mathfrak{l}, \mathfrak{r},\beta)$ is a representation of a Hom-alternative algebra $A$ if and only if the direct sum $(A \oplus V,\star,\alpha+\beta)$ of vector spaces is a Hom-alternative algebra (the semi-direct product) by defining multiplication in $A\oplus V$ by
$$(x+a)\star (y+b)= x\ast y+\mathfrak{l}(x)b+ \mathfrak{r}(y)a,\ \ \forall\ x,y\in A, \ a,b \in V.$$
Next, for its associated Hom-Malcev admissible algebra $(A \oplus V, \overbrace{[-, -]}, \alpha+\beta)$,
\begin{align*}
 \overbrace{[x+a, y+b]}=&(x+a)\star (y+b) -(y+b)\star (x+a)\\
=& x\ast y+ \mathfrak{l}(x)b+ \mathfrak{r}(y)a - y\ast x -\mathfrak{l}(y)a - \mathfrak{r}(x)b \\
=&[x, y] + (\mathfrak{l}-\mathfrak{r})(x)b - (\mathfrak{l} -\mathfrak{r})(y)a.
\end{align*}
According to Proposition \ref{semidirectprduct HomMalcev}, $(V,\mathfrak{l}-\mathfrak{r},\beta)$ is a representation of $(A,[-, -],\alpha)$.
Moreover, $T$ is an $\mathcal{O}$-operator of $(A,[-,-],\alpha)$ with respect to $(V,\mathfrak{l}-\mathfrak{r},\beta)$ since
\begin{align*}
[T(a),T(b)]=&T(a)\ast T(b)-T(b)\ast T(a)\\
=&T \big(\mathfrak{l}(T (a))b + \mathfrak{r}(T (b))a\big)-T \big(\mathfrak{l}(T (b))a + \mathfrak{r}(T (a))b\big)\\
=&T \big((\mathfrak{l}-\mathfrak{r})(T (a))b - (\mathfrak{l}-\mathfrak{r})(T (b))a\big).
\qedhere \end{align*}
\end{proof}
\begin{thm}\label{Hom-pre-altToHom-Pre-Malcev}
    Let $T:V\to A$ be an $\mathcal{O}$-operator of Hom-alternative algebra $(A,\ast,\alpha)$ with respect to the bimodule $(V,\mathfrak{l},\mathfrak{r},\beta)$. Then $(V,\prec,\succ, \beta)$ be a Hom-pre-alternative algebra, where for all $a,b\in V$,
\begin{equation}
  \label{homalt==>prehomalt} a\succ b= \mathfrak{l}(T (a))b \ \ \text{and}\ \ a\prec b= \mathfrak{r}(T (b))a.
\end{equation}
 Moreover, if $(V,\cdot,\beta)$ is the  Hom-pre-Malcev  algebra associated to the  Hom-Malcev admissible algebra $(A,[-,-],\alpha)$ on the module $(V,\mathfrak{l}-\mathfrak{r},\beta)$, then $a\cdot b=a\succ b-b\prec a$.

\end{thm}

\begin{proof} For any $a,b,c\in V$, using \eqref{rephomalt3} and \eqref{O-ophomalternative} yields
\begin{align*}
&(a\succ b) \prec \beta(c) -\beta(a)\succ(b \prec c) + (b \prec a) \prec\beta(c)-\beta(b) \prec (a \ast c)\\
&\quad =\mathfrak{r}(T (\beta(c)))\mathfrak{l}(T (a))b-\mathfrak{l}(T(\beta(a)))\mathfrak{r}(T(c))b \\
&\quad\quad  +\mathfrak{r}(T(\beta(c)))\mathfrak{r}(T (a))b-\mathfrak{r}(T(a \ast c))\beta(b)\\
&\quad = \mathfrak{r}(\alpha(T(c)))\mathfrak{l}(T (a))b-\mathfrak{l}(\alpha(T(a)))\mathfrak{r}(T(c))b \\
&\quad \quad +\mathfrak{r}(\alpha(T(c)))\mathfrak{r}(T(a))b-\mathfrak{r}(T(a \ast c))\beta(b)=0.
\end{align*}
The other identities for $(V,\prec,\succ,\beta)$ being a Hom-pre-alternative algebras can be verified
similarly.
Moreover, using   \eqref{hommalcev==>hompremalcev} and  \eqref{homalt==>prehomalt},
\begin{equation*}
a\cdot b=(\mathfrak{l}-\mathfrak{r})(T(a))b=\mathfrak{l}(T(a))b-\mathfrak{r}(T(a))b=a\succ b-b\prec a.
\qedhere \end{equation*}
 \end{proof}
\begin{cor}\label{homalt==>homprealt}
Let $(A, \ast , \alpha)$ be a Hom-alternative algebra and $\mathcal{R}:A\rightarrow A$ be a Rota-Baxter operator
of weight $0$ such that $ \mathcal{R}\alpha =\alpha  \mathcal{R}$. If multiplications
$\prec $ and $\succ $ on $A$ are defined for all $x, y\in A$ by
$x\prec y = x\ast \mathcal{R}(y)$ and $x\succ y = \mathcal{R}(x)\ast y$,
then $(A, \prec , \succ , \alpha)$ is a Hom-pre-alternative algebra.

Moreover, if $(A,\cdot,\alpha)$ be the  Hom-pre-Malcev  algebra associated to the  Hom-Malcev admissible algebra $(A,[-,-],\alpha)$, then $x\cdot y=x\succ y-y\prec x$.
\end{cor}
Moreover, Hom-Malcev
algebras, Hom-alternative algebras, Hom-pre-Malcev algebras and
Hom-pre-alternative algebras are closely related as follows (in the sense of commutative diagram of categories):
\begin{equation}\label{Diagramme1}
\begin{split}
\resizebox{13cm}{!}{\xymatrix{
\ar[rr] \mbox{\bf Hom-pre-alt alg $(A,\prec,\succ,\alpha)$ }\ar[d]_{\mbox{ $\ast=\prec+\succ$}}\ar[rr]^{\mbox{\quad\quad $\cdot=\prec-\succ$\quad\quad }}
                && \mbox{\bf Hom-pre-Malcev alg $(A,\cdot ,\alpha)$ }\ar[d]_{\mbox{ Commutator}}\\
\ar[rr] \mbox{\bf Hom-alt alg $(A,\ast,\alpha)$}\ar@<-1ex>[u]_{\mbox{ R-B }}\ar[rr]^{\mbox{ Commutator\quad\quad}}
                && \mbox{\bf Hom-Malcev alg  $(A,[-,-],\alpha)$}\ar@<-1ex>[u]_{\mbox{ R-B}}}
}\end{split}
\end{equation}

\subsection{Bimodules and \texorpdfstring{$\mathcal{O}$-}operators of Hom-pre-Malcev algebras} \label{subsec:bimodules}


In this subsection, we introduce and study bimodules of Hom-pre-Malcev algebras.
\begin{defn}\label{representation HPM}
 Let $(A, \cdot, \alpha)$ be a Hom-pre-Malcev algebra and $V$ be a vector space. Let
  $\ell, r: A \longrightarrow  End(V)$ be two linear maps and $\beta \in End(V)$. Then $(V, \ell,  r, \beta)$ is called a bimodule of $A$  if the
following conditions hold for any $x,y,z \in A$:
\begin{align}
 &\beta \ell(x)=\ell(\alpha(x)) \beta,\quad \beta r(x)=r(\alpha(x)) \beta,\label{rep1}\\
\begin{split}
& r(\alpha^{2}(x))\rho(\alpha (y))\rho(z)- r(\alpha (z)\cdot(y\cdot x))\beta^{2} + \ell(\alpha^{2}(y))r(z\cdot x)\beta \\
&\quad \quad+ \ell(\alpha ([y,z]))r(\alpha (x))\beta - \ell(\alpha^{2}(z))r(\alpha (x))\rho(y) = 0,
\end{split}
\label{rep2}\\
\begin{split}
 &\ell(\alpha^{2} (y))\ell(\alpha (z))r(x)-r(\alpha^{2}(x))\rho(\alpha (y))\rho(z) - \ell(\alpha^{2}(z))r(y\cdot x)\beta \\
 &\quad \quad- r(\alpha (z\cdot x))\rho(\alpha (y))\beta+ r([z,y]\cdot\alpha (x))\beta^{2} = 0,
 \end{split}
 \label{rep3}\\
 \begin{split}
   &r(\alpha (y)\cdot(z\cdot x))\beta^{2}+r(\alpha^{2} (x))\rho([y,z])\beta-\ell(\alpha^{2} (y))\ell(\alpha (z)) r(x) \\
    &\quad \quad+r(\alpha (y\cdot x))\rho(\alpha (z))\beta +\ell(\alpha^{2}(z))r(\alpha (x))\rho(y)= 0,
    \end{split}
    \label{rep4}\\
     \begin{split}
     &\ell([[x,y],\alpha (z)])\beta^{2}- \ell(\alpha^{2} (x))\ell(\alpha (y))\ell(z) + \ell(\alpha^{2} (z))\ell(\alpha (x))\ell(y) \\
   &\quad \quad+\ell(\alpha ([y,z])\ell(\alpha (x))\beta+ \ell(\alpha^{2}(y))\ell([x,z])\beta = 0,
     \end{split}
   \label{rep5}
 \end{align}
 where $\rho(x)=\ell(x)-r(x)$ and $[x,y]= x\cdot y - y\cdot x$.
\end{defn}
Now, define a linear operation $\cdot_{\ltimes} : \otimes^{2}(A \oplus V ) \longrightarrow (A \oplus V )$ by
$$(x + a) \cdot_{\ltimes} (y + b) = x \cdot y + \ell(x)(b) + r(y)(a), ~~\forall x, y \in A, a, b \in V,$$
and a linear map $\alpha + \beta : A \oplus V \longrightarrow A \oplus V$ by
$$(\alpha + \beta)(x + a) = \alpha(x) + \beta(a), ~~\forall x \in A, a \in V.$$

\begin{prop}\label{semidirectproduct hompreMalcev}
With the above notations, $(A\oplus V, \cdot_{\ltimes}, \alpha+ \beta)$ is a  Hom-pre-Malcev algebra, which is
denoted by $A \ltimes_{(\ell,~r)}^{\alpha, \beta}V$ or simply $A \ltimes V$ and called the semi-direct product of the  Hom-pre-Malcev algebra $(A, \cdot, \alpha)$ and
the representation $(V, \ell, r, \beta)$.
\end{prop}
\begin{proof}
For any $x, y, z, t \in A$ and  $a, b, c, d \in V$,
\begin{align*}
&\big((\alpha + \beta)[y+b,z+c]_{\rho}\big)\cdot_{\ltimes} (\alpha + \beta) \big((x+a)\cdot_{\ltimes}(t + d)\big) \\
 & \quad =\big(\alpha([y,z]) + \rho(\alpha(y))\beta(c)-\rho(\alpha(z))\beta(b)\big)\cdot_{\ltimes}\big(\alpha(x\cdot t) + \ell(\alpha(x))d + r(\alpha(t))a\big)\\
&\quad=\alpha([y,z])\cdot\alpha (x\cdot t) + \ell(\alpha ([y,z]))\big(\ell(\alpha(x))\beta(d) + r(\alpha(t))\beta(a)\big)\\
&\quad\quad + r(\alpha(x\cdot t))\big(\rho(\alpha(y))\beta(c) -\rho(\alpha(z))\beta(b)\big),
\\
&[[x + a,y + b]_{\rho},(\alpha+\beta)(z+c)]_{\rho}\cdot_{\ltimes}(\alpha^{2}+\beta^{2})(t+ d)\\
&\quad=[[x,y]+ \rho(x)b -\rho(y)a,\alpha(z) +\beta(c)]_{\rho}\cdot_{\ltimes}(\alpha^{2}(t) +  \beta^{2}(d))\\
&\quad=\big([[x,y],\alpha(z)]+ \rho([x,y])\beta(c)-\rho(\alpha(z))(\rho(x)b -\rho(y)a)\big)\cdot_{\ltimes}(\alpha^{2}(t) + \beta^{2}(d))\\
&\quad=[[x,y],\alpha(z)]\cdot\alpha^{2}(t) + \ell([[x,y],\alpha(z)])\beta^{2}(d)\\
& \quad\quad+ r(\alpha^{2}(t))\big(\rho([x,y])\beta(c)-\rho(\alpha(z))(\rho(x)b -\rho(y)a)\big),
\\
&(\alpha^{2}+\beta^{2})(y+b)\cdot_{\ltimes}\big([x +a,z+c]_{\rho}\cdot_{\ltimes}(\alpha+\beta)(t+ d)\big)\\
&\quad=(\alpha^{2}(y)+\beta^{2}(b))\cdot_{\ltimes}\big(([x,z] + \rho(x)c -\rho(z)a)\cdot_{\ltimes}(\alpha(t) + \beta(d))\big)\\
&\quad=(\alpha^{2}(y)+\beta^{2}(b))\cdot_{\ltimes}\big([x,z]\cdot\alpha(t) +
\ell([x,z])\beta(d) + r(\alpha(t))(\rho(x)c-\rho(z)a)\big)\\
&\quad=\alpha^{2}(y)\cdot ([x,z]\cdot\alpha(t)) + \ell(\alpha^{2}(y))\big(\ell([x,z])\beta(d) + r(\alpha(t))(\rho(x)c-\rho(z)a)\big)\\
 &\quad\quad+ r\big([x,z]\cdot\alpha(t)\big)\beta^{2}(b),
\\
&(\alpha^{2}+\beta^{2})(x+a)\cdot_{\ltimes}\big((\alpha+\beta)(y+b)\cdot_{\ltimes}((z +c) \cdot_{\ltimes}(t+d))\big)\\
 &\quad= (\alpha^{2}(x)+\beta^{2}(a))\cdot_{\ltimes}\big((\alpha(y) + \beta(b))\cdot_{\ltimes}((z\cdot t) + \ell(z)d + r(t)c)\big)\\
 &\quad=(\alpha^{2}(x)+\beta^{2}(a))\cdot_{\ltimes}\big(\alpha(y)\cdot(z\cdot t)+ \ell(\alpha(y))(\ell(z)d + r(t)c) +
r(z\cdot t)\beta(b)\big) \\
&\quad=\alpha^{2}(x)\cdot (\alpha(y)\cdot(z\cdot t)) + \ell(\alpha^{2}(x))\big(\ell(\alpha(y))(\ell(z)d +r(t)c) +r(z\cdot t)\beta(b)\big)\\
&\quad\quad+ r\big(\alpha(y)\cdot(z\cdot t)\big)\beta^{2}(a),
\\
&(\alpha^{2}+\beta^{2})(z+c)\cdot_{\ltimes}\big((\alpha+\beta)(x+a)\cdot_{\ltimes}((y +b) \cdot_{\ltimes}(t+d))\big)\\
&\quad= (\alpha^{2}(z)+\beta^{2}(c))\cdot_{\ltimes}\big((\alpha(x) + \beta(a))\cdot_{\ltimes}((y\cdot t) + \ell(y)d + r(t)b)\big)\\
&\quad=(\alpha^{2}(z)+\beta^{2}(c))\cdot_{\ltimes}\big(\alpha(x)\cdot(y\cdot t)+ \ell(\alpha(x))(\ell(y)d +r(t)b) +
r(y\cdot t)\beta(a)\big) \\
&\quad=\alpha^{2}(z)\cdot (\alpha(x)\cdot(y\cdot t)) + \ell(\alpha^{2}(z))\big(\ell(\alpha(x))(\ell(y)d +r(t)b) +r(y\cdot t)\beta(a)\big)\\
& \quad\quad+ r\big(\alpha(x)\cdot(y\cdot t)\big)\beta^{2}(c).
\end{align*}
Hence $(A \oplus V, \cdot _{\ltimes}, \alpha + \beta)$ is a Hom-pre-Malcev algebra if and only if
$(V, \ell, r, \beta)$ is a bimodule of $(A,\cdot, \alpha)$.
\end{proof}
\begin{prop}\label{rephompremalcev==rephommalcev}
  Let $(V,\ell,r,\beta)$ be a   bimodule of a Hom-pre-Malcev algebra $(A,\cdot,\alpha)$  and $(A,[-, -],\alpha)$
  be its associated Hom-Malcev algebra. Then,
  \begin{enumerate}
    \item \label{item1:prop:rephompremalcev==rephommalcev}
    $(V,\ell,\beta)$ is a representation of $(A, [-, -],\alpha)$,
    \item \label{item2:prop:rephompremalcev==rephommalcev}
    $(V,\ell - r,\beta)$ is a representation of $(A, [-, -],\alpha)$.
  \end{enumerate}
\end{prop}
\begin{proof}
\ref{item1:prop:rephompremalcev==rephommalcev} The statement \ref{item1:prop:rephompremalcev==rephommalcev} follows immediately from \eqref{rep2}.\\
\ref{item2:prop:rephompremalcev==rephommalcev}  By Proposition \ref{semidirectproduct hompreMalcev},
$A\ltimes_{\ell,r}^{\alpha, \beta} V$ is a Hom-pre-Malcev algebra. For its associated Hom-Malcev
 algebra $(A \oplus V, \overbrace{[-, -]}, \alpha+\beta)$,
\begin{align*}
 \overbrace{[x+a, y+b]}& =(x+a)\cdot_{\ltimes} (y+b) - (y+b)\cdot_{\ltimes} (x+a) \\
&= x \cdot y+ \ell(x)b+ r(y)a - y \cdot x - \ell(y)a - r(x)b \\
&= [x, y] + (\ell-r)(x)b - (\ell -r)(y)a.
\end{align*}
By Proposition \ref{semidirectprduct HomMalcev}, $(V,\ell-r,\beta)$ is a representation of $(A,[-, -],\alpha)$.
\end{proof}


If $(A, \cdot, \alpha)$ is a Hom-pre-Malcev algebra and $(A, [-, -], \alpha)$ is the associated  Hom-Malcev algebra, then $(A, L_\cdot, \alpha)$ is a representation of $(A, [-, -], \alpha)$ , where $L_\cdot$ is the left operation of $(A, \cdot, \alpha)$ given by $L(x)(y)=x\cdot y$.
  \begin{prop}
  Let $(A,\cdot,\alpha)$ be a Hom-algebra.
  Then $(A,\cdot,\alpha)$ is a Hom-pre-Malcev algebra if and only if $(A, [-, -],\alpha)$ defined by  \eqref{commutator}
is a Hom-Malcev algebra and $(A,L_{\cdot},\alpha)$ is a representation of $(A, \cdot,\alpha)$.
\end{prop}
    \begin{proof}
       It follows from the definition of Hom-Malcev algebra and representation of  Hom-Malcev algebra. Then, for any $x,y,z,t\in A$,
    \begin{align*}
 &\alpha([y, z]) \cdot \alpha(x \cdot t) + [[x, y], \alpha(z)]\cdot \alpha^{2}(t) + \alpha^{2}(y)\cdot ([x, z] \cdot \alpha(t))\\& - \alpha^{2}(x)\cdot (\alpha(y)\cdot (z\cdot t)) + \alpha^{2}(z)\cdot (\alpha(x)\cdot (y\cdot t))\\
 &=  \Big( L_{\cdot}(\alpha[y,z])L_{\cdot}(\alpha(x))\alpha +  L_{\cdot}([[x,y],\alpha(z)])\alpha^{2} +  L_{\cdot}(\alpha^{2}(y)) L_{\cdot}([x,z])\alpha\\
  &- L_{\cdot}(\alpha^{2}(x))L_{\cdot}(\alpha(y))L_{\cdot}(z) + L_{\cdot}(\alpha^{2}(z))L_{\cdot}(\alpha(x))L_{\cdot}(x) \Big) (t) = 0.
\qedhere \end{align*}
    \end{proof}
As in \cite{Bai2,K2}, we rephrase the definition of $\mathcal{O}$-operator in terms of Hom-pre-Malcev
algebras as follows.
   \begin{defn}\label{o-ophpm}
Let $(A, \cdot, \alpha)$ be a Hom-pre-Malcev algebra and $(V,\ell,r,\beta)$ be a bimodule. A linear map $T : V \to  A $ is called an $\mathcal{O}$-operator associated to $(V,\ell,r,\beta)$
 if $T$ satisfies
 \begin{align}
 T\circ \beta &= \alpha \circ T,\\
 T (a) \cdot T (b)&= T \big(\ell(T (a))b + r(T (b))a\big), \quad\forall a, b \in V.
\end{align}
\end{defn}
\begin{rem}
Let $T$ is an $\mathcal{O}$-operator of a Hom-pre-Malcev algebra $(A, \cdot, \alpha)$ associated to $(V, \ell, r, \beta)$.
Then $T$ is an $\mathcal{O}$-operator of its associated Hom-Malcev algebra $(A, [-, -], \alpha)$ associated to $(V, \ell - r, \beta)$.
\end{rem}
\begin{proof}
By Proposition \ref{rephompremalcev==rephommalcev}, for all $a, b \in V,$
\begin{align*}
[T(a), T(b)]& = T(a) \cdot T(b) - T(b) \cdot T(a)\\
& = T (\ell(T (a))b + r(T (b))a) - T (\ell(T (b))a + r(T (a))b)\\
& = T ((\ell- r)(T (a))b - (\ell -r)(T (b))a).
\qedhere \end{align*}
\end{proof}
\section{Hom-M-dendriform algebras} \label{sec:Hom-M-dendriform algebras}
 The goal of this section is to introduce the notion of Hom-M-dendriform algebras which is the Hom-type of M-dendriform and show that is a Hom-pre-Malcev.
\begin{defn}
Hom-M-dendriform algebra is a vector space $A$ endowed with two
bilinear products $\blacktriangleright, \blacktriangleleft: A\times A\to A$ and a linear map
$\alpha: A \to A$
such that for all $x, y, z, t \in A$ and
\begin{align}
&x\cdot y=x \blacktriangleleft y+ x\blacktriangleright y,   \label{horizontal} \\
&x \diamond y= x \blacktriangleleft y -y \blacktriangleright x, \label{vertical} \\
&[x, y]=x\cdot y-y \cdot x=x\diamond y-y\diamond x, \label{horizontal-vertical}
\end{align}
the following identities are satisfied:
\begin{align}\label{dend1}
&
\begin{array}{l}
(\alpha(z)\diamond(y\diamond x))\blacktriangleright \alpha^{2}(t)- \alpha^{2}(x)\blacktriangleright (\alpha (y)\cdot(z\cdot t))+\alpha^{2}(z)\blacktriangleleft(\alpha(x)\blacktriangleright(y\cdot t)) \\
\quad + \alpha([y,z])\blacktriangleleft \alpha(x\blacktriangleright t)
-\alpha^{2}(y)\blacktriangleleft((z\diamond x)\blacktriangleright \alpha(t)) = 0,
\end{array}
\\
\label{dend2}
&
\begin{array}{l}
\alpha^{2}(z)\blacktriangleleft(\alpha(x)\blacktriangleleft(y\blacktriangleright t))
-(\alpha (z)\diamond(x\diamond y))\blacktriangleright \alpha^{2}(t)- \alpha^{2}(x)\blacktriangleleft (\alpha (y)\blacktriangleright(z\cdot t)) \\
\quad
 - \alpha(z\diamond y)\blacktriangleright \alpha(x\cdot t)+ \alpha^{2}(y)\blacktriangleright([x,z]\cdot \alpha(t))= 0,
\end{array} \\
\label{dend3}
&
\begin{array}{l}
\alpha^{2}(z)\blacktriangleleft(\alpha(x)\blacktriangleleft(y\blacktriangleleft t))+([x,y]\diamond\alpha(z))\blacktriangleright \alpha^{2}(t)- \alpha^{2}(x)\blacktriangleleft(\alpha(y)\blacktriangleleft(z\blacktriangleright t)) \\
\quad
+ \alpha(y\diamond z)\blacktriangleright \alpha(x\cdot t) + \alpha^{2}(y)\blacktriangleleft((x\diamond z)\blacktriangleright \alpha(t))= 0,
\end{array} \\
\label{dend4}
&
\begin{array}{l}
[[x,y],\alpha (z)]\blacktriangleleft \alpha^{2}(t)-\alpha^{2} (x)\blacktriangleleft(\alpha (y)\blacktriangleleft (z \blacktriangleleft t)) + \alpha^{2}(z)\blacktriangleleft(\alpha (x)\blacktriangleleft (y\blacktriangleleft t)) \\
\quad
+ \alpha ([y,z])\blacktriangleleft\alpha (x \blacktriangleleft t)+ \alpha ^{2}(y)\blacktriangleleft([x,z]\blacktriangleleft\alpha(t))= 0.
\end{array}
\end{align}
\end{defn}

\begin{thm}\label{product}
Let $(A, \blacktriangleright,\blacktriangleleft, \alpha)$ be a Hom-M-dendriform algebra.
\begin{enumerate}
         \item \label{thm:product:1} The product given by   \eqref{horizontal} defines a Hom-pre-Malcev algebra $(A,\cdot, \alpha)$, called the associated horizontal Hom-pre-Malcev algebras.
         \item \label{thm:product:2} The product given by   \eqref{vertical} defines a Hom-pre-Malcev algebra $(A,\diamond, \alpha)$, called the associated vertical Hom-pre-Malcev algebras.
         \item \label{thm:product:3} The associated horizontal and vertical Hom-pre-Malcev algebras $(A,\cdot, \alpha)$ and $(A,\diamond, \alpha)$ of a Hom-M-dendriform algebra $(A, \blacktriangleright,\blacktriangleleft, \alpha)$ have the same associated Hom-Malcev algebras $(A,[-,-],\alpha)$ defined by   \eqref{horizontal-vertical}, called the associated Hom-Malcev algebra of the  Hom-M-dendriform algebra $(A, \blacktriangleright,\blacktriangleleft, \alpha)$.
       \end{enumerate}
\end{thm}
\begin{proof}
We will just prove \ref{thm:product:1}. In fact, using \eqref{HPM}, for any $x,y,z,t\in A,$
\begin{align*}
&
\alpha([y, z])\cdot\alpha(x\cdot t)+[[x,y],\alpha(z)] \cdot \alpha^{2}(t)
+ \alpha^{2}(y) \cdot ([x,z] \cdot \alpha (t))\\
&\quad-\alpha^{2}(x) \cdot (\alpha (y) \cdot (z \cdot t))  + \alpha^{2}(z) \cdot (\alpha (x) \cdot (y \cdot t))\\
&=\alpha([y, z])\blacktriangleleft\alpha(x\blacktriangleleft t)+\alpha([y,z])\blacktriangleleft\alpha( x\blacktriangleright t) + \alpha([y, z])\blacktriangleright \alpha(x\blacktriangleleft t)\\
& \quad+\alpha([y,z])\blacktriangleright \alpha(x\blacktriangleright t)+[[x,y],\alpha(z)]\blacktriangleleft\alpha^{2}(t)  +[[x,y],\alpha(z)]\blacktriangleright\alpha^{2}(t)\\
&\quad+  \alpha^{2}(y)\blacktriangleleft([x,z]\blacktriangleleft\alpha(t))+ \alpha^{2}(y)\blacktriangleleft([x,z]\blacktriangleright\alpha(t))+\alpha^{2}(y)\blacktriangleright([x,z]\blacktriangleleft\alpha(t))\\
&\quad+\alpha^{2}(y)\blacktriangleright([x,z]\blacktriangleright\alpha(t))+\alpha^{2}(z)\blacktriangleleft(\alpha(x)\blacktriangleleft(y\blacktriangleleft t))+\alpha^{2}(z)\blacktriangleleft(\alpha(x)\blacktriangleleft(y\blacktriangleright t))\\
&\quad+ \alpha^{2}(z)\blacktriangleleft(\alpha(x)\blacktriangleright(y\blacktriangleleft t))+\alpha^{2}(z)\blacktriangleleft(\alpha(x)\blacktriangleright(y\blacktriangleright t))
+\alpha^{2}(z)\blacktriangleright(\alpha(x)\blacktriangleleft(y\blacktriangleleft t))\\
&\quad+ \alpha^{2}(z)\blacktriangleright(\alpha(x)\blacktriangleleft(y\blacktriangleright t))
 + \alpha^{2}(z)\blacktriangleright(\alpha(x)\blacktriangleright(y\blacktriangleleft t))
 +\alpha^{2}(z)\blacktriangleright(\alpha(x)\blacktriangleright(y\blacktriangleright t))\\
&\quad-\alpha^{2}(x)\blacktriangleleft(\alpha(y)\blacktriangleleft( z\blacktriangleleft t))
-\alpha^{2}(x)\blacktriangleleft(\alpha(y)\blacktriangleleft(z\blacktriangleright t))
-\alpha^{2}(x)\blacktriangleleft(\alpha(y)\blacktriangleright(z\blacktriangleleft t))\\
&\quad-\alpha^{2}(x)\blacktriangleleft(\alpha(y)\blacktriangleright(z\blacktriangleright t))
-\alpha^{2}(x)\blacktriangleright(\alpha(y)\blacktriangleleft(z\blacktriangleleft t))
-\alpha^{2}(x)\blacktriangleright(\alpha(y)\blacktriangleleft(z\blacktriangleright t))\\
&\quad-\alpha^{2}(x)\blacktriangleright(\alpha(y)\blacktriangleright(z\blacktriangleleft t))
-\alpha^{2}(x)\blacktriangleright(\alpha(y)\blacktriangleright(z\blacktriangleright t))\\
&=\alpha([y, z])\blacktriangleleft\alpha(x\blacktriangleright t)+(\alpha(z)\diamond(y\diamond x))\blacktriangleright\alpha^{2}(t)-\alpha^{2}(y)\blacktriangleleft((z\diamond x)\blacktriangleright\alpha(t))\\
&\quad+\alpha^{2}(z)\blacktriangleleft(\alpha(x)\blacktriangleright(y\cdot t))-\alpha^{2}(x)\blacktriangleright(\alpha(y)\cdot(z\cdot t))
-\alpha(z\diamond y)\blacktriangleright \alpha(x\cdot t)\\
&\quad-(\alpha(z)\diamond(x\diamond y))\blacktriangleright\alpha^{2}(t)+\alpha^{2}(y)\blacktriangleright([x,z]\cdot\alpha(t))
+\alpha^{2}(z)\blacktriangleleft(\alpha(x)\blacktriangleleft(y\blacktriangleright t))\\
&\quad-\alpha^{2}(x)\blacktriangleleft(\alpha(y)\blacktriangleright(z\cdot t))+\alpha(y\diamond z)\blacktriangleright \alpha(x\cdot t)
+([x,y]\diamond\alpha(z))\blacktriangleright\alpha^{2}(t)\\
&\quad+ \alpha^{2}(y)\blacktriangleleft((x\diamond z)\blacktriangleright\alpha(t))+\alpha^{2}(z)\blacktriangleright(\alpha(x)\cdot(y\cdot t))
-\alpha^{2}(x)\blacktriangleleft(\alpha(y)\blacktriangleleft(z\blacktriangleright t))\\
&\quad+[[x,y],\alpha(z)]\blacktriangleleft\alpha^{2}(t) + \alpha([y, z])\blacktriangleleft \alpha(x\blacktriangleleft t)
+\alpha^{2}(y)\blacktriangleleft([x,z]\blacktriangleleft\alpha(t))\\
&\quad+ \alpha^{2}(z)\blacktriangleleft(\alpha(x)\blacktriangleleft(y\blacktriangleleft t))-\alpha^{2}(x)\blacktriangleleft(\alpha(y)\blacktriangleleft( z\blacktriangleleft t))=0.
\qedhere \end{align*}
\end{proof}
Hom-M-dendriform algebras are closely related to bimodules for Hom-pre-Malcev algebras.
\begin{prop}
Let $(A,\blacktriangleright,\blacktriangleleft,\alpha)$ be a Hom-M-dendriform algebra.
Let $L_{\blacktriangleleft}$ and $R_{\blacktriangleright}$ be the left and right multiplication operators corresponding respectively to the two operations $\blacktriangleright$ and $\blacktriangleleft$.
Then $(A,L_{\blacktriangleleft},R_{\blacktriangleright},\alpha)$ is a bimodule of its associated horizontal Hom-pre-Malcev algebra $(A,\cdot,\alpha)$.
\end{prop}
\begin{proof} We  verify that   \eqref{rep2} and \eqref{rep5} hold for $(A,L_{\blacktriangleleft},R_{\blacktriangleright},\alpha)$.
For any $x,y,z,t \in A$,
\begin{align*}
& R_\blacktriangleright(\alpha^{2}(x))L_\diamond(\alpha (y))L_\diamond(z)(t)- R_\blacktriangleright(\alpha (z)\cdot(y\cdot x))\alpha^{2}(t) + L_\blacktriangleleft(\alpha^{2}(y))R_\blacktriangleright(z\cdot x)\alpha(t)\\
 & \quad + L_\blacktriangleleft(\alpha ([y,z]))R_\blacktriangleright(\alpha (x))\alpha(t)- L_\blacktriangleleft(\alpha^{2}(z))R_\blacktriangleright(\alpha (x))L_\diamond(y)(t)\\
& = (\alpha (y)\diamond (z\diamond t))\blacktriangleright\alpha^{2}(x)-\alpha^{2}(t)\blacktriangleright(\alpha (z)\cdot(y\cdot x))  + \alpha^{2}(y)\blacktriangleleft(\alpha(t)\blacktriangleright(z\cdot x))\\
 & \quad+ \alpha ([y,z])\blacktriangleleft \alpha (x\blacktriangleright t) - \alpha^{2}(z)\blacktriangleleft((y\diamond t)\blacktriangleright\alpha (x))= 0.
\\
   &L_\blacktriangleleft[[x,y],\alpha (z)]\alpha^{2}(t) - L_\blacktriangleleft(\alpha^{2} (x))L_\blacktriangleleft(\alpha (y))L_\blacktriangleleft(z)(t) + L_\blacktriangleleft(\alpha^{2}(z))L_\blacktriangleleft(\alpha (x))L_\blacktriangleleft(y)(t)\\
 & \quad+ L_\blacktriangleleft(\alpha ([y,z]))L_\blacktriangleleft(\alpha (x))\alpha(t) + L_\blacktriangleleft(\alpha ^{2}(y))L_\blacktriangleleft([x,z])\alpha(t)\\
& = [[x,y],\alpha (z)]\blacktriangleleft \alpha^{2}(t)-\alpha^{2} (x)\blacktriangleleft(\alpha (y)\blacktriangleleft (z \blacktriangleleft t)) + \alpha^{2}(z)\blacktriangleleft(\alpha (x)\blacktriangleleft (y\blacktriangleleft t))\\
& \quad+ \alpha ([y,z])\blacktriangleleft\alpha (x \blacktriangleleft t)+ \alpha ^{2}(y)\blacktriangleleft([x,z]\blacktriangleleft\alpha(t))= 0.
 \end{align*}
Other identities can be proved using similar computations. Thus, $(A,L_{\blacktriangleleft},R_{\blacktriangleright},\alpha)$ is a representation of Hom-pre-Malcev $(A,\cdot,\alpha)$.
\end{proof}

\begin{prop}
 Let $(A,\blacktriangleright,\blacktriangleleft,\alpha)$ be a Hom-M-dendriform algebra. With two binary operations $\blacktriangleright ^{t}, \blacktriangleleft ^{t}:
A\otimes A \to A$ defined for all $x, y\in A$ by
\begin{equation}\label{transpose}
 x \blacktriangleright ^{t} y = -y \blacktriangleright x, x \blacktriangleleft ^{t} y = x \blacktriangleleft y,
\end{equation}
$(A,\blacktriangleright^{t},\blacktriangleleft^{t},\alpha)$ is a Hom-M-dendriform algebra.

Its  associated horizontal Hom-pre-Malcev algebra is the associated vertical Hom-pre-Malcev algebra $(A, \diamond,\alpha)$ of $(A, \blacktriangleright,\blacktriangleleft,\alpha)$, and its associated vertical Hom-pre-Malcev
algebra is the associated horizontal Hom-pre-Malcev algebra $(A,\cdot,\alpha)$ of $(A,\blacktriangleright,\blacktriangleleft,\alpha)$, that is,
\begin{align}
\label{transpose point} x\cdot^t y &=x\blacktriangleleft^t y+x\blacktriangleright^t y=x\blacktriangleleft
  y-y\blacktriangleright y=x\diamond y,\\
\label{tanspose diamond} x\diamond^t y &=x\blacktriangleleft^t y-y\blacktriangleright^t x=x\blacktriangleleft y+x\blacktriangleright y=x\cdot y,\\
\begin{split} [x,y]^{t} &=x\blacktriangleleft^t y+x\blacktriangleright^t y-y\blacktriangleleft^t x-y\blacktriangleright^t x \\
 &=x\blacktriangleleft y-y\blacktriangleright x-y\blacktriangleleft x+x\blacktriangleright y=[x,y].
\label{transpose bracket}
\end{split}
\end{align}
\end{prop}
\begin{proof}
By \eqref{transpose point}-\eqref{transpose bracket}, for all $x,y,z,t\in A$,
\begin{align*}
&[[x,y]^{t},\alpha (z)]^{t}\blacktriangleleft^{t} \alpha^{2}(t)-\alpha^{2} (x)\blacktriangleleft^{t}(\alpha (y)\blacktriangleleft^{t} (z \blacktriangleleft^{t} t)) + \alpha^{2}(z)\blacktriangleleft^{t}(\alpha (x)\blacktriangleleft^{t} (y\blacktriangleleft^{t} t))\\
&\quad+\alpha ([y,z]^{t})\blacktriangleleft^{t}\alpha (x \blacktriangleleft^{t} t)+ \alpha ^{2}(y)\blacktriangleleft^{t}([x,z]^{t}\blacktriangleleft^{t}\alpha(t))\\
&=[[x,y],\alpha (z)]\blacktriangleleft \alpha^{2}(t)-\alpha^{2} (x)\blacktriangleleft(\alpha (y)\blacktriangleleft (z \blacktriangleleft t)) + \alpha^{2}(z)\blacktriangleleft(\alpha (x)\blacktriangleleft (y\blacktriangleleft t))\\
&\quad+ \alpha ([y,z])\blacktriangleleft\alpha (x \blacktriangleleft t)+ \alpha ^{2}(y)\blacktriangleleft([x,z]\blacktriangleleft\alpha(t)).
\end{align*}
Similarly, \eqref{dend1}$^{t}$=\eqref{dend1}, \eqref{dend2}$^{t}$=\eqref{dend2}  and \eqref{dend3}$^{t}$=\eqref{dend3}. Therefore, $(A,\blacktriangleright^t,\blacktriangleleft^t,\alpha)$ is a Hom-M-dendriform algebra.
\end{proof}
\begin{defn} Let $(A,\blacktriangleright,\blacktriangleleft,\alpha)$ be a Hom-M-dendriform algebra. The Hom-M-dendri\-form algebra $(A,\blacktriangleright^{t},\blacktriangleleft^{t},\alpha)$
given by \eqref{transpose} is called the transpose of $(A,\blacktriangleright,\blacktriangleleft,\alpha)$.
\end{defn}

Examples of Hom-M-dendriform algebras can be constructed from Hom-pre-Malcev algebras
with $\mathcal{O}$-operators.
For brevity, we only give the study involving the associated
horizontal Hom-pre-Malcev algebras.
\begin{prop}
 Let $(A, \cdot, \alpha)$ be a Hom-pre-Malcev algebra and $(V, \ell, r, \beta)$ be a bimodule of $A$.  Let $T$ be an $\mathcal{O}$-operator of $(A, \cdot, \alpha)$  associated to $(V, \ell, r, \beta)$.
Then $(V, \blacktriangleright, \blacktriangleleft, \beta)$ is a Hom-M-dendriform algebra, where
for all $a, b \in V$,
\begin{equation}\label{hpm==>dend}
~a \blacktriangleright b = r(T(b))a, \quad a\blacktriangleleft b = \ell(T(a))b.
\end{equation}
Therefore, there is a Hom-pre-Malcev algebra on $V$ given in Theorem \ref{product} as the associated horizontal Hom-pre-Malcev algebra of $(V,\blacktriangleright,\blacktriangleleft, \beta)$, and $T$ is a morphism of Hom-pre-Malcev algebras. Moreover, $T(V) = \{T(v) \mid \; v \in V\} \subset A $ is a Hom-pre-Malcev subalgebra of $(A, \cdot,\alpha)$, and there is an induced Hom-M-dendriform  on $(T(V ),\triangleright,\triangleleft,\alpha)$ given for all $a, b \in V$ by
\begin{align*}
    T(a) \triangleright T(b)=T(a\blacktriangleright b),\quad  T(a) \triangleleft T(b)=T(a \blacktriangleleft b).
\end{align*}
Its corresponding associated horizontal Hom-pre-Malcev algebraic
structure on $T(V )$ is just the subalgebra of the Hom-pre-Malcev $(A, \cdot,\alpha)$, and $T$ is a
homomorphism of Hom-M-dendriform algebras.
\end{prop}
\begin{proof}
For any $a, b, c, d \in V$, using \eqref{rep2} and Definition \ref{o-ophpm},
\begin{align*}
& (\beta(c)\diamond(b\diamond a))\blacktriangleright \beta^{2}(d)- \beta^{2}(a)\blacktriangleright (\beta (b)\cdot(c\cdot d))+\beta^{2}(c)\blacktriangleleft(\beta(a)\blacktriangleright(b\cdot d))\\
 & \quad+ \beta([b,c])\blacktriangleleft \beta(a\blacktriangleright d)-\beta^{2}(b)\blacktriangleleft((c\diamond a)\blacktriangleright \beta(d))\\
&=(\beta (c)\blacktriangleleft(b\blacktriangleleft a))\blacktriangleright\beta^{2}(d)-((b\blacktriangleleft a)\blacktriangleright \beta(c))\blacktriangleright \beta^{2}(d)- (\beta(c)\blacktriangleleft(a\blacktriangleright b))\blacktriangleright \beta^{2}(d)\\
&\quad+((a\blacktriangleright b)\blacktriangleright \beta(c))\blacktriangleright \beta^{2}(d)- \beta^{2}(a)\blacktriangleright (\beta (b)\blacktriangleleft(c\blacktriangleleft d))- \beta^{2}(a)\blacktriangleright(\beta (b)\blacktriangleright(c\blacktriangleleft d))\\
&\quad- \beta^{2}(a)\blacktriangleright(\beta (b)\blacktriangleleft(c\blacktriangleright d)) - \beta^{2}(a)\blacktriangleright(\beta(b)\blacktriangleright(c\blacktriangleright d))
 + \beta^{2}(c)\blacktriangleleft(\beta(a)\blacktriangleright(b\blacktriangleleft d))\\
 &\quad+ \beta^{2}(c)\blacktriangleleft(\alpha(a)\blacktriangleright(b\blacktriangleright d))+ \beta(b\blacktriangleleft c)\blacktriangleleft \beta(a\blacktriangleright d) + \beta(b\blacktriangleright c)\blacktriangleleft\beta(a\blacktriangleright d)\\
 &\quad- \beta(c\blacktriangleleft b)\blacktriangleleft\beta(a\blacktriangleright d)-\beta(c\blacktriangleright b)\blacktriangleleft \beta(a\blacktriangleright d)
 -\beta^{2}(b)\blacktriangleleft((c\blacktriangleleft a)\blacktriangleright \beta(d))\\
 &\quad+ \beta^{2}(b)\blacktriangleleft((a\blacktriangleright c)\blacktriangleright \beta(d))\\
 &= r(\alpha^{2}(T(d)))\ell(\alpha(T(c)))\ell(T(b))a- r(\alpha^{2}(T(d)))r(\alpha(T(c)))\ell(T(b))a \\
 &\quad- r(\alpha^{2}(T(d)))\ell(\alpha(T(c)))r(T(b))a+r(\alpha^{2}(T(d)))r(\alpha(T(c)))r(T(b))a\\
 &\quad-r(T(\ell(T(\beta(b)))\ell(T(c))d))\beta^{2}(a) - r(T(r(T(\ell(T(c))d))\beta(b)))\beta^{2}(a)\\
 &\quad-r(T(\ell(T(\beta(b)))r(T(d))c))\beta^{2}(a)- r(T(r(T(r(T(d))c))\beta(b)))\beta^{2}(a)\\
 &\quad+ \ell(\alpha^{2}(T(c)))r(T(\ell(T(b))d))\beta(a) + \ell(\alpha^{2}(T(c)))r(T(r(T(d))b))\beta(a)\\
 &\quad+ \ell(\alpha(T(\ell(T(b))c)))r(\alpha(T(d))a)+\ell(\alpha(T(r(T(c))b)))r(\alpha(T(d))a)\\
 &\quad- \ell(\alpha(T(\ell(T(c))b)))r(\alpha(T(d))a)-\ell(\alpha(T(r(T(b))c)))r(\alpha(T(d))a)\\
 &\quad- \ell(\alpha^{2}(T(b)))r(\alpha(T(d)))\ell(T(c))a +\ell(\alpha^{2}(T(b)))r(\alpha(T(d)))r(T(c))a\\
&=r(\alpha^{2}(T(d)))\ell(\alpha(T(c)))\rho(T(b))a- r(\alpha^{2}(T(d)))r(\alpha(T(c)))\rho(T(b))a \\
&\quad-r(T(\ell(T(\beta(b)))(T(c)\cdot T(d))))\beta^{2}(a)-r(T(r(T(c)\cdot T(d))\beta(b)))\beta^{2}(a)\\
&\quad+\ell(\alpha^{2}(T(c)))r(T(b)\cdot T(d))\beta(a) +\ell(\alpha(T(b)\cdot T(c)))r(\alpha(T(d))a)\\
&\quad-\ell(\alpha(T(c)\cdot T(b)))r(\alpha(T(d))a) - \ell(\alpha^{2}(T(b)))r(\alpha(T(d)))\rho(T(c))a\\
& = r(\alpha^{2}(T(d)))\rho(\alpha(T(c)))\rho(T(b))a -r(\alpha(T(b))\cdot(T(c)\cdot T(d)))\beta^{2}(a)\\
 &\quad+\ell(\alpha^{2}(T(c)))r(T(b)\cdot T(d))\beta(a)+\ell(\alpha([T(b),T(c)]))r(\alpha(T(d))a)\\
 &\quad- \ell(\alpha^{2}(T(b)))r(\alpha(T(d)))\rho(T(c))a =0.
\end{align*}
Using a similar computation,    \eqref{dend2}-\eqref{dend4}  can be checked. Therefore, $(V, \blacktriangleright, \blacktriangleleft, \beta)$ is a Hom-M-dendriform algebra.
The other conclusions follow easily. \qedhere
\end{proof}
Now, we introduce the following concept of Rota-Baxter operator on a Hom-pre-Malcev algebra which is a particular case of $\mathcal{O}$-operator.
 \begin{defn}
 Let $(A, \cdot, \alpha)$ be a Hom-pre-Malcev algebra. A linear map $\mathcal{R}: A \longrightarrow A$ is
called a Rota-Baxter operator of weight zero on A if for all $x, y \in A$,
$$\mathcal{R}\circ\alpha = \alpha\circ \mathcal{R}, \quad \quad
\mathcal{R}(x)\cdot \mathcal{R}(y) = \mathcal{R}\big(\mathcal{R}(x)\cdot y + x\cdot \mathcal{R}(y)\big).$$
\end{defn}
\begin{cor}\label{HPM==>HMD by rota-baxter}
Let $(A, \cdot, \alpha)$ be a Hom-pre-Malcev algebra and $\mathcal{R} : A \to A$ be a Rota-Baxter operator of weight $0$ for $A$. Define new operations on $A$ by
$$x \blacktriangleright  y = x \cdot \mathcal{R}(y), \quad \quad
x \blacktriangleleft y = \mathcal{R}(x) \cdot y.$$
Then $(A,\blacktriangleright, \blacktriangleleft, \alpha)$ is a Hom-M-dendriform algebra.
\end{cor}
\begin{lem}\label{commuting rotabaxter}
Let $\mathcal{R}_1$ and $\mathcal{R}_2$ be two commuting Rota-Baxter operators $($of weight
zero$)$ on a Hom-Malcev algebra $(A, [-, -],\alpha)$. Then $\mathcal{R}_2$ is a Rota-Baxter operator $($of weight zero$)$ on the Hom-pre-Malcev algebra $(A,\cdot,\alpha)$, where for all $x,y \in A$,
$$x \cdot y= [\mathcal{R}_1(x),  y].$$
\end{lem}

\begin{proof}
For any $x,y \in A$,
\begin{align*}
 \mathcal{R}_2(x) \cdot \mathcal{R}_2(y)=& [\mathcal{R}_1(\mathcal{R}_2(x)), \mathcal{R}_2(y)] \\
=& \mathcal{R}_2([\mathcal{R}_1(\mathcal{R}_2(x)), y] + [\mathcal{R}_1(x), \mathcal{R}_2(y)]) \\
=& \mathcal{R}_2(\mathcal{R}_2(x) \cdot y + x \cdot \mathcal{R}_2(y)).
\qedhere \end{align*}
\end{proof}

\begin{cor}
Let $\mathcal{R}_1$ and $\mathcal{R}_2$ be two commuting Rota-Baxter operators
$($of weight zero$)$ on a Hom-Malcev algebra $(A, [-, -],\alpha)$. Then there exists a Hom-M-dendriform algebra structure on $A$ given for all $x,y \in A$ by
\begin{align}\label{ rota-baxter HMD}
    x \blacktriangleright y= [\mathcal{R}_1(x),  \mathcal{R}_2(y)],\quad   x \blacktriangleleft y=[\mathcal{R}_1(\mathcal{R}_2(x)),  y].
\end{align}
\end{cor}

\begin{proof}
By Lemma \ref{commuting rotabaxter}, $\mathcal{R}_2$ is a Rota-Baxter operator of weight zero on $(A,\cdot,\alpha)$, where
$$x \cdot y=[ \mathcal{R}_1(x), y].$$
Then, applying Corollary \ref{HPM==>HMD by rota-baxter}, there exists a Hom-M-dendriform algebraic structure on $A$ given for all $x,y \in A$ by
\begin{align*}
   & x \blacktriangleright y=x\cdot \mathcal{R}_2(y)= [\mathcal{R}_1(x), \mathcal{R}_2(y)],\\
    &  x \blacktriangleleft y= \mathcal{R}_2(x)\cdot y=[\mathcal{R}_1(\mathcal{R}_2(x)), y].
\qedhere \end{align*}
\end{proof}
\begin{ex}
Consider the $4$-dimensional Hom-Malcev algebra and the Rota-Baxter operators $\mathcal{R}$ in Example \ref{ex:4dmalcev}.
Then, there is a Hom-M-dendriform algebraic structure on $A$ given for all $x,y \in A$ by
$$x \blacktriangleright_{\alpha} y= \alpha([\mathcal{R}(x),  \mathcal{R}(y)]),\quad \quad  x \blacktriangleleft_{\alpha} y= \alpha([\mathcal{R}^2(x),y]),$$  that is
$$
\begin{array}{c|cccc}
  \blacktriangleright_{\alpha}  & e_1 & e_2 & e_3 &  e_4 \\
  \hline
  e_1& 0 & \lambda_1 e_3   & 0 & 0 \\
  e_2 & -\lambda_1 e_3  &   0 & 0 & 0\\
  e_3 & 0 & 0  & 0 & 0\\
  e_4 &  0 & 0  & 0 & 0
  \end{array} \quad \quad \quad
\begin{array}{c|cccc}
  \blacktriangleleft_{\alpha}  & e_1 & e_2 & e_3 &  e_4 \\
  \hline
 e_1 & \frac{a_4}{2}e_4  & -\alpha(e_2) & e_3 & -e_4 \\
  e_2 & 0  & 0 &  0 & 0\\
  e_3 & 0 & 0  & 0 & 0\\
  e_4 &  0 & 0  & 0 & 0
  \end{array}.$$
\end{ex}
\begin{ex}
Consider the $5$-dimensional Hom-Malcev algebra and Rota-Baxter operators $\mathcal{R}$ in Example \ref{ex:5dmalcev}.
Then, there is a Hom-M-dendriform algebraic structure on $A$ given for all $x,y \in A$ by
$$x \blacktriangleright_{\alpha} y= \alpha([\mathcal{R}(x),  \mathcal{R}(y)]),\quad \quad
x \blacktriangleleft_{\alpha} y=\alpha([\mathcal{R}^2(x),  y]), $$
that is
$$
\begin{array}{c|ccccc}
  \blacktriangleright_{\alpha}  & e_1 & e_2 & e_3 &  e_4  & e_5\\
  \hline
  e_1& 0 & 0  & 0 & b e_3 & 0 \\
  e_2 & 0  &   0 & 0 & 0 & 0\\
  e_3 & 0 & 0  & 0 & 0 & 0\\
  e_4 & -b e_3 & 0  & 0 & 0 & 0\\
  e_5 & 0 & 0  & 0 & 0 & 0
  \end{array} \quad \quad \quad
\begin{array}{c|ccccc}
  \blacktriangleleft_{\alpha}  & e_1 & e_2 & e_3 &  e_4 & e_5 \\
  \hline
 e_1 & -a_4e_2  &- a_5e_3 & 0 & e_2 & \frac{- b a_4 }{a_5}e_3 \\
  e_2 & 0  & 0 &  0 & 0 & 0\\
  e_3 & 0 & 0  & 0 & 0 & 0\\
  e_4 &  0 & 0  & 0 & 0 & 0\\
  e_5 &  0 & 0  & 0 & 0 & 0
  \end{array}.$$
\end{ex}
Hom-M-dendriform algebras are related to Hom-alternative quadri-algebras in the same way Hom-L-dendriform algebras are related to Hom-quadri-algebras (see \cite{ChtiouiMabroukMakhlouf2} for more details).
\begin{defn}\label{def:altquad}
    Hom-alternative quadri-algebra is a 6-tuple $(A, \nwarrow, \swarrow, \nearrow, \searrow, \alpha)$
consisting of a vector  space $A$, four bilinear maps $\nwarrow,
\swarrow, \nearrow, \searrow: A \times A\rightarrow A$ and a linear map $\alpha
: A\rightarrow A$ which is algebra morphism such that the following axioms are satisfied for all $x, y, z\in A:$
    \begin{alignat*}{4}
       \las x,y,z \ras^{r}_{\alpha} + \las y,x,z \ras^{m}_{\alpha} &= 0, &
        \qquad \qquad
        \las x,y,z \ras^{r}_{\alpha} + \las x,z,y \ras^{r}_{\alpha} &= 0,
        \\
        \las x,y,z \ras^{n}_{\alpha} + \las y,x,z \ras^{w}_{\alpha} &= 0, &
        \qquad \qquad
        \las x,y,z \ras^{n}_{\alpha} + \las x,z,y \ras^{ne}_{\alpha} &= 0,
        \\
        \las x,y,z \ras^{ne}_{\alpha} + \las y,x,z \ras^{e}_{\alpha} &= 0, &
        \qquad \qquad
        \las x,y,z \ras^{w}_{\alpha} + \las x,z,y \ras^{sw}_{\alpha} &= 0,
        \\
        \las x,y,z \ras^{sw}_{\alpha} + \las y,x,z \ras^{s}_{\alpha} &= 0, &
        \qquad \qquad
        \las x,y,z \ras^{m}_{\alpha} + \las x,z,y \ras^{\ell}_{\alpha} &= 0,
        \\
        \las x,y,z \ras^{\ell}_{\alpha} + \las y,x,z \ras^{\ell}_{\alpha} &= 0, &
    \end{alignat*}
where
    \begin{alignat*}{4}
        \las x,y,z \ras^{r}_{\alpha} &= ( x \nwarrow y ) \nwarrow \alpha(z) - \alpha(x) \nwarrow ( y \ast z ) && \quad \text{\rm (right $\alpha$-associator)} \\
        \las x,y,z \ras^{\ell}_{\alpha} &= ( x \ast y ) \searrow \alpha(z) - \alpha(x) \searrow ( y \searrow z ) && \quad \text{\rm (left $\alpha$-associator)} \\
        \las x,y,z \ras^{m}_{\alpha} &= ( x \searrow y ) \nwarrow \alpha(z) - \alpha(x) \searrow ( y \nwarrow z )  && \quad \text{\rm (middle $\alpha$-associator)}\\
        \las x,y,z \ras^{n}_{\alpha} &= ( x \nearrow y ) \nwarrow \alpha(z) - \alpha(x) \nearrow ( y \prec z ) && \quad \text{\rm (north $\alpha$-associator)} \\
        \las x,y,z \ras^{w}_{\alpha} &= ( x \swarrow y ) \nwarrow \alpha(z) - \alpha(x) \swarrow ( y \wedge z) && \quad \text{\rm (west $\alpha$-associator)} \\
        \las x,y,z \ras^{s}_{\alpha} &= ( x \succ y ) \swarrow \alpha(z) - \alpha(x) \searrow ( y \swarrow z ) && \quad \text{\rm (south $\alpha$-associator)} \\
        \las x,y,z \ras^{e}_{\alpha} &= ( x \vee y ) \nearrow \alpha(z) - \alpha(x) \searrow ( y \nearrow z ) && \quad \text{\rm (east $\alpha$-associator)} \\
        \las x,y,z \ras^{ne}_{\alpha} &= ( x \wedge y ) \nearrow \alpha(z) - \alpha(x) \nearrow ( y \succ z ) && \quad \text{\rm (north-east $\alpha$-associator)} \\
        \las x,y,z \ras^{sw}_{\alpha} &= ( x \prec y ) \swarrow \alpha(z) - \alpha(x) \swarrow ( y \vee z ) && \quad \text{\rm (south-west $\alpha$-associator)}
   \end{alignat*}
   \begin{alignat*}{3}
   x \succ y&= x \nearrow y+ x \searrow y ,  & x \prec y= x \nwarrow y + x \swarrow y, \\
   x \vee y&= x \searrow y+ x \swarrow y,   & x \wedge y= x \nearrow y + x \nwarrow y, \\
        x \ast y   &= x \succ y + x \prec y &= x \searrow y + x \nearrow y + x \nwarrow y + x \swarrow y.
    \end{alignat*}
\end{defn}
\begin{lem}
    Let $(A,\nearrow, \searrow, \swarrow, \nwarrow,\alpha)$ be some Hom-alternative quadri-algebra. Then
    $(A,\prec,\succ,\alpha)$ and $(A,\vee,\wedge,\alpha)$ are Hom-pre-alternative algebras
    (called respectively horizontal and vertical Hom-pre-alternative structures associated to $A$),
    and $(A,\ast,\alpha)$ is a Hom-alternative algebra.
\end{lem}
Definition \ref{def:bimodhomprealt} introduces the notion of bimodule of Hom-pre-alternative algebras.
\begin{defn} \label{def:bimodhomprealt}
Let $( A,\prec,\succ,\alpha)$ be a Hom-pre-alternative algebra and $(V,\beta)$ a vector space. Let $L_\succ,R_\succ,L_\prec,,R_\prec:  A\to gl(V)$ be linear maps.
Then, $(V ,L_\succ,R_\succ,L_\prec,,R_\prec,\beta)$ ia called a
representation or a bimodule of $( A\prec,\succ,\alpha)$ if for any $x, y \in A$,
\begin{eqnarray}
 L_\succ(x\ast y+y\ast x)\beta=L_\succ(\alpha(x))L_\succ(y)+L_\succ(\alpha(y))L_\succ(x),\label{pabm1}\\
R_\succ(\alpha(y))(L(x)+R(x))=L_\succ(\alpha(x))R_\succ(y)+R_\succ(x\succ y)\beta,\\
R_\prec(\alpha(y))L_\succ(x)+R_\prec(\alpha(y))R_\prec(x)=L_\succ(\alpha(x))R_\prec(y)+R_\prec(x\ast y)\beta,\\
R_\prec(\alpha(y))R_\succ(x)+R_\succ(\alpha(y))L_\prec(x)=L_\prec(\alpha(x))R(y)+R_\succ(x\ast y)\beta,\\
L_\prec(y\prec x)\beta+L_\prec(x\succ y)\beta=L_\prec(\alpha(y))L(x)+L_\succ(\alpha(y))L_\succ(x),\\
R_\prec(\alpha(x))L_\succ(y)+L_\succ(y\succ x)\beta=L_\succ(y)R_\prec(x)+L_\succ(\alpha(y))L_\succ(x),\\
R_\prec(\alpha(x))R_\succ(y)+R_\succ(\alpha(y))R(x)=R_\succ(y\prec x)\beta+R_\succ(x\succ y)\beta,\\
L_\prec(y\succ x)\beta+R_\succ(\alpha(x))L(y)=L_\succ(\alpha(y))L_\prec(x)+L_\succ(\alpha(y))R_\succ(y),\\
R_\prec(\alpha(x))R_\prec(y)+R_\prec(\alpha(y))R_\prec(x)=R_\prec(x\ast y+y\ast x)\beta,\\
R_\prec(\alpha(y))L_\prec(x)+L_\prec(x\prec y)\beta=L_\prec(\alpha(x))(R(y)+L(y)),\label{pabm10}
\end{eqnarray}
where $x\ast y=x\prec y+y\succ x$, $L=L_\prec+L_\succ$ and $R=R_\prec+R_\succ$.
\end{defn}
\begin{prop}\label{directsumhomprealt}
A tuple
$(V,L_\succ,R_\succ,L_\prec,R_\prec,\beta)$ is a bimodule of a Hom-pre-alter\-native algebra
$( A,\prec,\succ,\alpha)$  if and only if the direct sum $( A\oplus V, \ll,\gg,\alpha+\beta)$ is a Hom-pre-alternative algebra, where for any $x,y \in  A, a,b \in V$,
\begin{align*}
  (x+a)\ll (y+b) &= x\prec y+L_\prec(x)b+R_\prec(y)a,\\
   (x+a)\gg (y+b) &= x\succ y+ L_\succ(x)b+R_\succ(x)a, \\
 (\alpha\oplus\beta)(x+a) &= \alpha(x)+\beta(a).
\end{align*}
We denote it by $A \ltimes^{\alpha,\beta}_{L_\succ,R_\succ,L_\prec,R_\prec} V$ or simply $A \ltimes V$.
\end{prop}
\begin{defn}
Let $(A,\prec,\succ,\alpha)$ be a Hom-pre-alternative algebra. Then a linear map $T: V \to A$ is called an $\mathcal{O}$-operator of  $(A,\prec,\succ,\alpha)$ associated to a bimodule   $(V,L_{\prec},R_{\prec},L_{\succ},R_{\succ},\beta)$  if $T$ satisfies, for all $a, b \in V$,
 \begin{align}\label{Oophomprealt}
\begin{split}
T \circ\beta&=\alpha\circ T,\\
 T(a) \succ T(b)&=T\big(L_{\succ}(T(a))b+R_{\succ}(T(b))a\big),\\
  T(a) \prec T(b)&=T\big(L_{\prec}(T(a))b+R_{\prec}(T(b))a\big).
\end{split}
\end{align}
\end{defn}
\begin{prop}\label{last prop}
Let $(A,\prec,\succ,\alpha)$ be some Hom-pre-alternative algebra, and let $T$ be  an $\mathcal{O}$-operator of  $(A,\prec,\succ,\alpha)$ associated to a bimodule $(V,L_{\prec},R_{\prec},L_{\succ},R_{\succ},\beta)$.  With products defined, for any $a,b \in V$, by
\begin{align}
\begin{split}
 a\se b &=L_{\succ}(T(a))b,\ a \ne b=R_{\succ}(T(b))a, \\
 a \sw b &= L_{\prec}(T(a))b,\ a \nw b=R_{\prec}(T(b))a,
 \end{split}
\end{align}
$(V,\se,\ne,\sw,\nw,\beta)$ is a Hom-alternative quadri-algebra.

\end{prop}
\begin{proof}
Set $L=L_{\prec}+L_{\succ}$ and $R=R_{\prec}+R_{\succ}$.  For any $a,b,c \in V$,
\begin{align*}
    &(a \ast b)\se \beta(c)-\beta(a) \se (b \se c) \\
    &\quad =(L(T(a))b+R(T(b))a)\se \beta(c)-\beta(a)\se (L_{\succ}(T(b))c) \\
    &\quad =L_{\succ}(T(L(T(a))b+R(T(b))a))\beta(c)-L_{\succ}(\alpha(T(a)))L_{\succ}(T(b))c \\
    &\quad =L_{\succ}(T(a)\ast T(b)) \beta(c)-L_{\succ}(\alpha(T(a)))L_{\succ}(T(b))c=0.
\\
& (a \sw b) \nw \beta(c)-\phi(a) \sw (b \nw c+ b \ne c) \\
& \quad =R_{\prec}(\alpha(T(c)))L_{\prec}(T(a))b-L_{\prec}(\alpha(T(a)))(R_{\prec}(T(c))b+R_{\succ}(T(c))b)\\
& \quad =R_{\prec}(\alpha(T(c)))L_{\prec}(T(a))b-L_{\prec}(\alpha(T(a)))R(T(c))b, \\
& (a \nw c+a \sw c ) \sw \beta(b)-\beta(a) \sw (c \se b+ c \sw b)\\
& \quad= L_{\prec}(T(R_{\prec}(T(c))a+L_{\prec}(T(a))c))\beta(b)- L_{\prec}(\alpha(T(a)))(L_{\succ}(T(c))b+L_{\prec}(T(c))b) \\
& \quad=L_{\prec}(T(a) \prec T(c))\beta(b)-L_{\prec}(\alpha(T(a)))(L(T(c))b.
\end{align*}
This means that
\begin{align*}
  \las a,b,c \ras^{w}_{\beta}+\las a,c,b \ras^{sw}_{\beta} =&
  R_{\prec}(\alpha(T(c)))L_{\prec}(T(a))b-L_{\prec}(\alpha(T(a)))R(T(c))b\\
 &+L_{\prec}(T(a) \prec T(c))\beta(b)-L_{\prec}(\alpha(T(a)))L(T(c))b=0,
\end{align*}
since $(V,L_{\prec},R_{\prec},L_{\succ},R_{\succ},\beta)$ is a bimodule of $(A,\prec,\succ,\alpha)$.
The rest of identities can be proved using  analogous  computations.\end{proof}

Analogously to what happens for quadri-algebras \cite{AL}, Rota-Baxter operators allow different constructions
for alternative quadri-algebras which is a particular case of $\mathcal{O}$-operator.
\begin{defn}
Let $(A, \prec , \succ , \alpha)$ be a Hom-pre-alternative
algebra. A Rota-Baxter operator of weight $0$ on $A$ is a linear map $\mathcal{R}:
A\rightarrow A$ such that $\mathcal{R}\alpha =\alpha  \mathcal{R}$ and the following conditions are satisfied, for all $x, y\in A$,
\begin{eqnarray}
&& \mathcal{R}(x)\succ \mathcal{R}(y)=\mathcal{R}(x\succ \mathcal{R}(y)+ \mathcal{R}(x)\succ y),  \label{baxter1}\\
&& \mathcal{R}(x)\prec \mathcal{R}(y)=\mathcal{R}(x\prec \mathcal{R}(y)+ \mathcal{R}(x)\prec y).  \label{baxter2}
\end{eqnarray}
\end{defn}
\begin{cor} \label{operation}
Let $(A, \prec , \succ , \alpha)$ be a Hom-pre-alternative
algebra and $\mathcal{R}: A\rightarrow A$ be a Rota-Baxter operator of weight $0$ for $A$.
Then $(A,\nearrow, \searrow, \swarrow, \nwarrow,\alpha)$ is a Hom-alternative quadri-algebra with the operations
    \begin{eqnarray*}
        x \nearrow y = x \succ \mathcal{R}(y),
        \qquad \qquad
        x \searrow y = \mathcal{R}(x) \succ y,
        \\
        x \swarrow y = \mathcal{R}(x) \prec y,
        \qquad \qquad
        x \nwarrow y = x \prec \mathcal{R}(y).
\end{eqnarray*}
\end{cor}
\begin{thm} Assume hypothesis of  Corollary \ref{homalt==>homprealt} and
let  $(V,L_{\prec},R_{\prec},L_{\succ},R_{\succ},\beta)$  be a bimodule and $T$ be an $\mathcal{O}$-operator of  $(A,\prec,\succ,\alpha)$ associated to  $(V,L_{\prec},R_{\prec},L_{\succ},R_{\succ},\beta)$. Then   $(V,L_{\succ}-R_{\prec},R_{\succ}-L_{\prec},\beta)$ is a bimodule of the Hom-pre-Malcev algebra $(A,\cdot,\alpha)$,   and $T$ is an $\mathcal{O}$-operator of $(A,\cdot,\alpha)$ with respect to $(V,L_{\succ}-R_{\prec},R_{\succ}-L_{\prec},\beta)$.  Moreover, if $(V,\blacktriangleright,\blacktriangleleft,\beta)$ is the  Hom-M-dendriform  algebra associated to the  Hom-pre-Malcev algebra $(A,\cdot,\alpha)$ on the bimodule $(V,L_{\succ}-R_{\prec},R_{\succ}-L_{\prec},\beta)$, then
 \[     a\blacktriangleright b = a \nearrow b - b \swarrow a,
 \qquad
        a  \blacktriangleleft  b = a \searrow b - b \nwarrow a.
    \]
\end{thm}
\begin{proof}
By Proposition \ref{directsumhomprealt}, $A\ltimes V$ is a Hom-pre-alternative algebra. Consider its associated Hom-pre-Malcev algebra $(A \oplus V, \bullet, \alpha+\beta)$,
\begin{align*}
& (x+a)\bullet (y+b) =(x+a)\gg (y+b) - (y+b)\ll (x+a) \\
&= x \succ y+ L_{\succ}(x)b+ R_{\succ}(y)a - y \prec x - L_{\prec}(y)a - R_{\prec}(x)b \\
&= x\cdot y + (L_{\succ}-R_{\prec})(x)b + ( R_{\succ}-L_{\prec})(y)a.
\end{align*}
By Proposition \ref{semidirectproduct hompreMalcev}, $(V,L_{\succ}-R_{\prec},R_{\succ}-L_{\prec},\beta)$ is a representation of Hom-pre-Malcev algebra $(A,\cdot,\alpha)$.
Moreover,\ $T$ is an $\mathcal{O}$-operator of $(A,\cdot,\alpha)$ with respect to $(V,L_{\succ}-R_{\prec},R_{\succ}-L_{\prec},\beta)$. In fact,
\begin{align*}
T(a)\cdot T(b)=&T(a)\succ T(b)-T(b)\prec T(a)\\
=&T \big(L_{\succ}(T (a))b + R_{\succ}(T (b))a\big)-T \big(L_{\prec}(T (b))a + R_{\prec}(T (a))b\big)\\
=&T \big((L_{\succ}-R_{\prec})(T (a))b + (R_{\succ}-L_{\prec})(T (b))a\big).
\end{align*}
Moreover, by \eqref{hpm==>dend} and \eqref{Oophomprealt},
\begin{align*}
&a\blacktriangleright b=(R_{\succ}-L_{\prec})(T(b))a=R_{\succ}(T(b))a-L_{\prec}(T(b))a=a\nearrow b - b \swarrow a,\\
&a \blacktriangleleft b =(L_{\succ}-R_{\prec})(T(a))b= L_{\succ}(T(a))b-R_{\prec}(T(a))b=a \searrow b - b \nwarrow a.
\qedhere \end{align*}
\end{proof}
\begin{cor}
Assume hypothesis of  Corollary \ref{homalt==>homprealt}, and let $\mathcal{R}$ be a Rota-Baxter operator of  $(A,\prec,\succ,\alpha)$ and $(A,\blacktriangleright,\blacktriangleleft,\alpha)$ be the  Hom-M-dendriform  algebra associated to the  Hom-pre-Malcev algebra $(A,\cdot,\alpha)$ given in Corollary \ref{HPM==>HMD by rota-baxter}. Then, the operations
 \[     x\blacktriangleright y = x \nearrow y - y \swarrow x,
 \qquad
        x  \blacktriangleleft  y = x \searrow y - y \nwarrow x,
    \]
 define a Hom-M-dendriform structure in $A$ with respect the twisting map $\alpha$.
\end{cor}
\noindent
Combining it with diagram \eqref{Diagramme1}, yields the following detailed commutative diagram:

\begin{equation}\label{diagramhommalcev}
    \begin{split}
\resizebox{14cm}{!}{\xymatrix{
\ar[rr] \mbox{\bf Hom-alt quadri-alg $(A,\nwarrow, \swarrow, \nearrow, \searrow,\alpha)$ }\ar[d]_{\mbox{$\begin{array}{l} \succ=\nearrow +\searrow \\ \prec=\nwarrow+ \swarrow\end{array}$}}\ar[rr]^{\mbox{ \quad$\blacktriangleright= \nearrow -\swarrow$}}_{\mbox{ \quad$\blacktriangleleft=\searrow -\nwarrow$}}
                && \mbox{\bf  Hom-M-dendri alg  $(A,\blacktriangleright,\blacktriangleleft,\alpha)$ }\ar[d]_{\mbox{$\cdot=\blacktriangleright+\blacktriangleleft$}}\\
\ar[rr] \mbox{\bf Hom-pre-alt alg $(A,\prec,\succ,\alpha)$}\ar@<-1ex>[u]_{\mbox{R-B }}\ar[d]_{\mbox{ $\star=\prec+\succ$}}\ar[rr]^{\mbox{\quad\quad $\cdot=\prec-\succ$\quad\quad  }}
                && \mbox{\bf Hom-pre-Malcev alg  $(A,\cdot,\alpha)$}\ar@<-1ex>[u]_{\mbox{R-B}}\ar[d]_{\mbox{Commutator}}\\
          \ar[rr] \mbox{\bf Hom-alt alg $(A,\star,\alpha)$}\ar@<-1ex>[u]_{ \mbox{R-B}}\ar[rr]^{\mbox{Commutator}}
                && \mbox{\bf Hom-Malcev alg  $(A,[-,-],\alpha)$}\ar@<-1ex>[u]_{\mbox{R-B}}}
}
 \end{split}
\end{equation}

\section{Twistings}\label{sec:twistings}
In \cite{Yau}, D. Yau shows that one can construct a Hom-Malcev algebra starting
from a Malcev algebra and an algebra morphism. Therefore using this construction, one can construct Hom-pre-Malcev algebras, Hom-M-Dendriform algebras and their representations.

\begin{thm}
Let $(A, [-, -])$ be a Malcev algebra and $\alpha: A \longrightarrow A$ be an algebra morphism. Then
the Hom-algebra $A_{\alpha} = (A, [-, -]_{\alpha} = \alpha\circ[-, -], \alpha)$ induced by $\alpha$ is a Hom-Malcev algebra.
Moreover, assume that $(A', [-, -]')$ is another Malcev algebra, and $\alpha': A' \to A'$ is a
Malcev algebra morphism. Let $f : A \to A'$ be a Malcev algebra
morphism satisfying $f \circ \alpha = \alpha'\circ f$. Then $f : A_{\alpha} \to A'_{\alpha'}$
 is a Hom-Malcev algebra morphism.
\end{thm}
The following result gives a procedure to construct representation of Hom-Malcev algebra by
a representation  Malcev algebra,  morphism and linear map.
\begin{prop}
   Let $(A, [-, -])$ be a Malcev algebra, $\alpha: A\to A$ be a morphism on $A$,  $(V, \rho)$ be a representation of $A$ and $\beta:V\to V$ be a linear map such that
 $\beta \rho(x)(a)=\rho(\alpha(x))\beta(a).$
 Then, $(V, \widetilde{\rho},\beta)$ is a representation of $(A, [-, -]_\alpha,\alpha)$, where
 $$\widetilde{\rho}(x)(a)=\rho(\alpha(x))\beta(a),\quad \forall \ x\in A,\ a\in V.$$
 \end{prop}
 \begin{proof}
 For any $x, y, z \in A$ and $a \in V$, by \eqref{rephommalcev},
 $$\beta \widetilde{\rho}(x)(a)=\beta(\rho(\alpha(x))\beta(a))=\rho(\alpha^{2}(x))\beta^{2}(a)=\widetilde{\rho}(\alpha(x))\beta(a).$$
Combining with \eqref{representation H-M}, we get
\begin{multline*}
\widetilde{\rho}([[x, y]_\alpha, \alpha (z)]_\alpha)(\beta^{2}(a))-  \widetilde{\rho}(\alpha^{2}(x))\widetilde{\rho}(\alpha (y))\widetilde{\rho}(z)(a) +\widetilde{\rho}(\alpha^{2}(z))\widetilde{\rho}(\alpha (x))\widetilde{\rho}(y)(a)\\
+ \widetilde{\rho}(\alpha^{2}(y)) \widetilde{\rho}([x,z]_\alpha)\beta(a) + \widetilde{\rho}(\alpha ([y,z]_\alpha))\widetilde{\rho}(\alpha (x))\beta(a) = \\
 \beta^{3}\Big(\rho([[x, y], z]) - \rho(x)\rho(y)\rho(z) + \rho(z)\rho(x)\rho(y) +\rho(y)\rho([x,z]) +\rho([y, z])\rho(x)\Big)(a) = 0.
 \end{multline*}
Then $(V, \widetilde{\rho},\beta)$ is a representation of $(A, [-, -]_\alpha,\alpha)$ on $V$.
\end{proof}
The following theorem provides a procedure to construct Hom-pre-Malcev algebra by
a pre-Malcev algebra and a morphism.
 \begin{thm}
 Let $\mathcal{A}=(A, \cdot)$ be a pre-Malcev algebra and let $\alpha :A\longrightarrow A$ be a morphism of $\mathcal{A}$. Then $\mathcal{A}_{\alpha} =(A, \cdot_{\alpha}, \alpha)$ is a  Hom-pre-Malcev algebra with $x \cdot_{\alpha} y = \alpha (x \cdot y).$
 Moreover, assume that $\mathcal{A'}=(A', \cdot')$ is another pre-Malcev algebra and $\alpha':A' \to A'$ is a
 pre-Malcev algebra morphism of $\mathcal{A'}$. Let $f : A \to A'$ be a pre-Malcev algebra
morphism satisfying $f \circ \alpha = \alpha'\circ f$. Then $f: A_{\alpha} \to A'_{\alpha'}$
 is a Hom-pre-Malcev algebra morphism.
 \end{thm}
 \begin{proof}
 We just show that $(A, \cdot_{\alpha}, \alpha)$ satisfies the
identity \eqref{HPM} while $(A, \cdot)$
satisfies the  identity \eqref{PM}. Indeed,
\begin{align*}
\alpha([y,z]_{\alpha})\cdot_{\alpha} \alpha(x \cdot_{\alpha} t) &= \alpha^{3}\big([y,z]\cdot (x\cdot t)\big), \\
[[x,y]_{\alpha},\alpha(z)]_{\alpha} \cdot_{\alpha}\alpha^{2}(t)&= \alpha^{3}([[x,y], z]\cdot t),\\
\alpha^{2}(y) \cdot_{\alpha} ([x,z]_{\alpha}\cdot_{\alpha} \alpha (t)) &= \alpha^{3}(y\cdot ([x,z]\cdot t)), \\
  \alpha^{2}(x) \cdot_{\alpha} (\alpha (y)  \cdot_{\alpha} (z \cdot_{\alpha} t))&=\alpha^{3}(x\cdot (y \cdot (z \cdot t))), \\
\alpha^{2}(z) \cdot_{\alpha} (\alpha (x) \cdot_{\alpha} (y \cdot_{\alpha} t)) &= \alpha^{3}(z\cdot (x\cdot(y\cdot t))).
\end{align*}
The second assertion follows from
\begin{equation*}
f(x\cdot_\alpha y) =f(\alpha(x\cdot y)) = \alpha'(f(x\cdot y))=\alpha'(f(x)\cdot' f(y))= f(x)\cdot_{\alpha'}'f(y).
\qedhere
\end{equation*}
\end{proof}
The following result gives a construction of a bimodule of  Hom-pre-Malcev algebra starting
with a classical one by means of the Yau twist procedure.
\begin{prop}
   Let $(A, \cdot)$ be a pre-Malcev algebra, $\alpha: A\to A$ be a morphism on $A$,  $(V, \ell,  r)$ be a bimodule of $(A, \cdot)$ and $\beta: V\to V$ be a linear map such that
 $$\beta \ell(x)(b)= \ell(\alpha(x))(\beta(b)), \quad  \beta r(x)(b)=r(\alpha(x))(\beta(b)).$$
 Then $(V, \widetilde{\ell},  \widetilde{r},\beta)$ is a bimodule of $(A, \cdot_\alpha,\alpha)$, where
 $$\widetilde{\ell}(x)(b)=\ell(\alpha(x))(\beta(b)), \quad  \widetilde{r}(x)(b)=r(\alpha(x))(\beta(b)).$$
 \end{prop}
   \begin{proof}
For  $x,y,z\in A, b \in V$,
\begin{align*}
&\beta\widetilde{\ell}(x)(b)=\beta(\ell(\alpha(x))(\beta(b)))=\ell(\alpha^{2}(x))\beta^{2}(b)=\widetilde{\ell}(\alpha(x))\beta(b),\\
&\beta\widetilde{r}(x)(b)=\beta(r(\alpha(x))(\beta(b)))=r(\alpha^{2}(x))\beta^{2}(b)=\widetilde{r}(\alpha(x))\beta(b),
\end{align*}
that is  \eqref{rep1} holds for $\widetilde{\ell}$ and $\widetilde{r}$. Using    \eqref{rep1} and \eqref{rep2}, we get
\begin{multline*}
 \widetilde{r}(\alpha^{2}(x))\widetilde{\rho}(\alpha (y))\widetilde{\rho}(z)(b)- \widetilde{r}(\alpha (z)\cdot_{\alpha}(y\cdot_{\alpha} x))\beta^{2}(b) + \widetilde{\ell}(\alpha^{2}(y))\widetilde{r}(z\cdot_{\alpha} x)\beta(b)\\
+ \widetilde{\ell}(\alpha ([y,z]_{\alpha}))\widetilde{r}(\alpha (x))\beta(b) - \widetilde{\ell}(\alpha^{2}(z))\widetilde{r}(\alpha (x))\widetilde{\rho}(y)(b)=\\
 \beta^{3}\Big(r(x)\rho(y)\rho(z)- r(z\cdot(y\cdot x)) + \ell(y)r(z\cdot x)\beta+ \ell([y,z])r(x)-\ell(z)r(x)\rho(y)\Big)(b) =0.
 \end{multline*}
 Similarly,    \eqref{rep3}-\eqref{rep5} also hold for $(V, \widetilde{\ell},  \widetilde{r},\beta).$
 \end{proof}
 Now, we provide a way to construct Hom-M-dendriform algebras starting
from an M-dendriform algebra and an algebra endomorphism.
\begin{prop}
Let $(A, \blacktriangleright,\blacktriangleleft )$ be an M-dendriform algebra and $\alpha : A \to A$ be an algebra morphism. Then
$(A, \blacktriangleright_{\alpha}, \blacktriangleleft_{\alpha}, \alpha)$ is a Hom-M-dendriform algebra where
$$x \blacktriangleright_{\alpha} y = \alpha(x \blacktriangleright y),  \quad x\blacktriangleleft_{\alpha} y= \alpha(x\blacktriangleleft y).$$
\end{prop}
\begin{proof}
We only prove that $(A, \blacktriangleright_{\alpha}, \blacktriangleleft_{\alpha}, \alpha)$ satisfies the first Hom-M-dendriform identity. The other identities for $(A, \blacktriangleright_{\alpha}, \blacktriangleleft_{\alpha}, \alpha)$ being a Hom-M-dendriform algebra can be verified
similarly. In fact, for any $x, y,z,t \in A$,
\begin{multline*}
(\alpha(z)\diamond_{\alpha} (y\diamond_{\alpha}  x))\blacktriangleright_{\alpha}  \alpha^{2}(t)- \alpha^{2}(x)\blacktriangleright_{\alpha}  (\alpha (y)\cdot_{\alpha} (z\cdot_{\alpha}  t))+\alpha^{2}(z)\blacktriangleleft_{\alpha} (\alpha(x)\blacktriangleright_{\alpha} (y\cdot_{\alpha}  t)) \\
\quad\quad\quad + \alpha([y,z]_{\alpha} )\blacktriangleleft_{\alpha}  \alpha(x\blacktriangleright_{\alpha}  t)-\alpha^{2}(y)\blacktriangleleft_{\alpha} ((z\diamond_{\alpha}  x)\blacktriangleright_{\alpha}  \alpha(t))\\
=\alpha^{3}\Big((z\diamond(y\diamond x))\blacktriangleright t-x\blacktriangleright (y\cdot(z\cdot t))+z\blacktriangleleft(x\blacktriangleright(y\cdot t))+[y,z]\blacktriangleleft(x\blacktriangleright t)\Big. \\
\Big. -y\blacktriangleleft((z\diamond x)\blacktriangleright t)\Big) =  0.
\end{multline*}
Hence  $(A, \blacktriangleright_{\alpha}, \blacktriangleleft_{\alpha}, \alpha)$ is a Hom-M-dendriform algebra.
\end{proof}

\end{document}